\documentclass[envcountsect,envcountsame]{svjour3}
\smartqed

\usepackage{amsmath}
\usepackage{amssymb}
\usepackage{xcolor}
\usepackage{bbm}
\usepackage{graphicx}
\usepackage{verbatim}

\usepackage[T1]{fontenc}
\usepackage[ansinew]{inputenc}
\usepackage[german, french,USenglish]{babel}

\makeatletter

\newcommand{\esssup}{\mathop{\mathrm{ess\,sup}}\displaylimits}

\usepackage[colorinlistoftodos,bordercolor=orange,backgroundcolor=orange!20,linecolor=orange,textsize=scriptsize]{todonotes}

\numberwithin{equation}{section}

\def\JELname{{\bfseries JEL Classification}\enspace}
      \def\JEL#1{\par\addvspace\medskipamount{\rightskip=0pt plus1cm
      \def\and{\ifhmode\unskip\nobreak\fi\ $\cdot$
      }\noindent\JELname\ignorespaces#1\par}}
\begin{document}
\title{From optimal martingales to randomized dual optimal stopping }
\subtitle{~}
\author{
Denis Belomestny
\and
John Schoenmakers
}

\institute{Denis Belomestny\at Faculty of Mathematics, University of Duisburg--Essen, Thea Leymann Str. 9, 45127 Essen, Germany \\
\email{denis.belomestny@uni-due.de
}
\and{John Schoenmakers\at WIAS Berlin, Germany\\
\email{john.schoenmakers@wias-berlin.de}}
}

\date{}
\maketitle
\begin{abstract}
In this article we study and classify  optimal martingales in the dual
formulation of optimal stopping problems. In this respect we distinguish
between weakly optimal and surely optimal martingales. It is shown that the family of
weakly optimal and surely optimal martingales may  be quite large. On the
other hand it is shown that the Doob-martingale, that is, the martingale part of
the Snell envelope, is in a certain sense the most robust surely optimal
martingale under random perturbations. This new insight leads to a novel
randomized dual martingale minimization algorithm that doesn't require nested
simulation. As a main feature, in a possibly large family of optimal
martingales the algorithm efficiently selects  a martingale that is as close as possible
to the Doob martingale. As a result, one obtains the dual
upper bound for the optimal stopping problem with low variance.

\JEL{G10\and G12\and G13}
\end{abstract}

\section{Introduction}

The last decades have seen a huge development of numerical methods for solving
optimal stopping problems. Such problems became very prominent in the
financial industry in the form of American derivatives. For such derivatives
one needs to evaluate the right of exercising (stopping) a certain cash-flow
(reward) process $Z$ at some (stopping) time $\tau$, up to some time horizon
$T$. From a mathematical point of view this evaluation comes down to solving
an optimal stopping problem
\[
Y^{\star}=\sup_{\text{stopping time\ }\tau\leq T}\mathsf{E[}\underset
{\text{reward at stopping}}{Z_{\tau}}].
\]
Typically the cash-flow $Z$ depends on various underlying assets and/or
interest rates and as such is part of a high dimensional Markovian framework.
Particularly for high dimensional stopping problems, virtually all generic
numerical solutions are Monte Carlo based. Most of the first numerical
solution approaches were of primal nature in the sense that the goal was to
construct a \textquotedblleft good\textquotedblright\ exercise policy and to
simulate a lower biased estimate of $Y^{\star}.$ In this respect we mention,
for example, the well-known regression methods by Longstaff \& Schwartz
\cite{J_LS2001}, Tsiklis \& Van Roy \cite{J_TV2001}, and the stochastic mesh
approach by Broadie \& Glasserman \cite{J_BrGl}, and the stochastic policy improvement
method by Kolodko \& Schoenmakers \cite{J_KS2006}. For further references we
refer to the literature, for example \cite{Gl} and the references therein.

In this paper we focus on the dual approach developed by Rogers
\cite{J_Rogers2002}, and Haugh \& Kogan \cite{J_HK2004}, initiated earlier by
Davis \& Karatzas \cite{J_DK1994}. In the dual method the stopping problem is
solved by minimizing over a set of martingales, rather than a set of stopping
times,
\begin{equation}
\label{In1}Y^{\star}=\inf_{M:\text{ martingale, }M_{0}=0\text{ }}%
\mathsf{E}\left[  \max_{0\leq s\leq T}\left(  Z_{s}-M_{s}\right)  \right]  .
\end{equation}
A canonical minimizer of this dual problem is the martingale part, $M^{\star}
$  of the Doob(-Meyer) decomposition of the Snell envelope
\[
Y_{t}^{\star}=\sup_{t\leq\text{stopping time }\tau\leq T}\mathsf{E}%
_{\mathcal{F}_{t}}\left[  Z_{\tau}\right]  ,
\]
which moreover has the nice property that
\begin{equation}
\label{In2}Y_{0}^{\star}=\max_{0\leq s\leq T}(Z_{s}-M_{s}^{\star})\text{
\ \ almost surely.}%
\end{equation}
That is, if one would succeed in finding $M^{\star}$, the value of $Y^{\star}
$ can be obtained from one trajectory of $Z-M^{\star}$ only.

Shortly after the development of the duality method in \cite{J_Rogers2002} and
\cite{J_HK2004}, various numerical approaches for computing dual upper bounds
for American options based on it appeared. May be one of the most popular
methods is the nested simulation approach by Andersen \& Broadie
\cite{J_AB2004}, who essentially construct an approximation to the Doob
martingale of the Snell envelope via stopping times obtained by the Longstaff
\& Schwartz method \cite{J_LS2001}. A few years later, a linear Monte Carlo
method for dual upper bounds was proposed in \cite{J_BelBenSch}. In fact, as a
common feature, both \cite{J_AB2004} and \cite{J_BelBenSch} aimed at
constructing (an approximation of) the Doob martingale of the Snell envelope
via some approximative knowledge of continuation functions obtained by the
method of Longstaff \& Schwartz or in another way. Instead of relying on such
information, the common goal in later studies \cite{J_DesFarMoa}, \cite{J_SchZhaHua}, \cite{J_Bel}, 
\cite{J_BelHilSch}, was to minimize the
expectation functional in the dual representation (\ref{In1}) over a linear
space of generic ``elementary'' martingales. Indeed, by parameterizing the
martingale family in a linear way and replacing the expectation in (\ref{In1})
by the sample mean over a large set of trajectories, the resulting
minimization comes down to solving a linear program. However, it was pointed
out in \cite{J_SchZhaHua} that in general there may exist martingales that are ``weakly'' optimal
in the sense that they minimize (\ref{In1}), but fail to have the ``almost
sure property'' (\ref{In2}). As a consequence, the estimator for the dual
upper bound due to such martingales may have high variance. Moreover, an
example in \cite{J_SchZhaHua} illustrates that a straightforward minimization of the
sample mean corresponding to (\ref{In1}) may end up with a martingale that is
asymptotically optimal in the sense of (\ref{In1}) but not surely optimal in
the sense of (\ref{In2}), when the sample size tends to infinity. As a remedy
to this problem, in \cite{J_Bel} variance penalization is proposed, whereas in
\cite{J_BelHilSch} the sample mean is replaced by the maximum over all trajectories.
\par
In this paper we first extend the study of surely optimal martingales in \cite{J_SchZhaHua}
to  the larger
class of {\em weakly} optimal
martingales. As a principal contribution, we give a
complete characterization of weakly and surely optimal martingales and moreover consider
the notion of randomized dual martingales. In particular, it is shown that in
general there may be a fullness of martingales that are optimal but not
surely optimal. In fact, straightforward minimization procedures based on
the sample mean in (\ref{In1}) may typically return martingales of this kind, even
if the Doob martingale of the Snell envelope is contained in the martingale family (as
illustrated already in \cite{J_SchZhaHua}, though at a somewhat pathological example
with partially deterministic cash-flows). As another main contribution we will show
that the Doob martingale plays a distinguished  role within the family of all optimal martingales.
Namely, it will be shown that by
randomizing the arguments in the path-wise maximum for each trajectory in a particular way,
any non-Doob optimal martingale can be turned to a suboptimal one. More specifically,
we will prove that there exists a particular ``optimal randomization'' such
that the Doob martingale, perturbed or randomized with it, remains guaranteed (surely)
optimal, while any other surely or weakly optimal martingale turns to a suboptimal one.
Of course, as a rule this ``optimal
randomization'' is not directly known or available in practical applications. But, it turns out that by
just incorporating some simple randomization due to uniform random variables,
 sample mean minimization  may return a martingale that is closer to the
Doob-martingale than one obtained without randomization. We thus end up with a martingale with low variance, which in
turn guarantees that the corresponding upper bound based on (\ref{In1}) is
tight (see \cite{J_Bel} and\cite{J_SchZhaHua}). Compared to \cite{J_BelHilSch} and
\cite{J_Bel}, the benefit of this new randomized dual approach is its
computational efficiency:
From the experiments we conclude that it may be sufficient to add on for each trajectory
simple i.i.d. uniform random variables to (some of) the arguments of the maximum.
An extensive numerical analysis of the here presented randomized dual martingale approach
will certainly be an interesting subsequent study but is
considered beyond the scope of this article.

The structure of the paper is as follows. Section~\ref{SecCh} carries out a
systematic theoretical analysis of optimal martingales. In
Section~\ref{SecRan} we deal with randomized optimal martingales and the effect
of randomizing the Doob-martingale. More technical proofs are given in
Section~\ref{SecProofs} and some first numerical examples are presented in
Section~\ref{SecNum}.

\section{Characterization of optimal martingales}

\label{SecCh}

Since practically any numerical approach to optimal stopping is based on a
discrete exercise grid, we will work within in a discrete time setup. That is,
it is assumed that exercise (or stopping) is restricted to a discrete set of
exercise times $t_{0}=0,$ $...,$ $t_{J}=T,$ for some time horizon $T$ and some
$J\in\mathbb{N}_{+}.$ For notational convenience we will further identify the
exercise times $t_{j}$ with their index $j,$ and thus monitor the reward
process $Z_{j},$ at the ``times'' $j=0,$ $...,$ $J.$

Let $(\Omega,\mathcal{F},\mathrm{P})$ be a filtered probability space with
discrete filtration $\mathcal{F}=(\mathcal{F}_{j})_{j\geq0}.$ An optimal
stopping problem is a problem of stopping the reward process $(Z_{j})_{j\geq
0}$ in such a way that the expected reward is maximized. The value of the
optimal stopping problem with horizon $J$ at time $j\in\{0,\ldots,J\}$ is
given by
\begin{equation}
Y_{j}^{\star}=\esssup_{{\tau\in\mathcal{T}[j,\ldots,J]}} \mathsf{E}%
_{\mathcal{F}_{j}}[Z_{\tau}], \label{eq:stop-prim}%
\end{equation}
provided that $Z$ was not stopped before $j.$ In (\ref{eq:stop-prim}),
$\mathcal{T}[j,\ldots,J]$ is the set of $\mathcal{F}$-stopping times taking
values in $\{j,\ldots,J\}$ and the process $\left(  Y_{j}^{\star}\right)
_{j\geq0}$ is called the Snell envelope. It is well known that $Y^{\star}$ is
a supermartingale satisfying the backward dynamic programming equation
(Bellman principle):
\[
Y_{j}^{\star}=\max\left(  Z_{j},\mathsf{E}_{\mathcal{F}_{j}}[Y_{j+1}^{\star
}]\right)  ,\text{ \ \ }0\leq j<J,\text{ \ \ }Y_{J}^{\star}=Z_{J}.
\]
Along with a primal approach based on the representation \eqref{eq:stop-prim},
a dual method was proposed in \cite{J_Rogers2002} and \cite{J_HK2004}. Below
we give a short self contained recap while including the notions of {\em weak} and {\em sure} optimality. 

Let $\mathcal{M}$ be the set of martingales $M$ adapted to $\mathcal{F}$ with
$M_{0}=0.$ By using the Doob's optimal sampling theorem one observes that
\begin{equation}
\label{2up}Y_{j}^{\star}\leq\mathsf{E}_{\mathcal{F}_{j}}\left[  \max_{j\leq
r\leq J}\left(  Z_{r}-M_{r}+M_{j}\right)  \right]  ,\text{ \ }j=0,\ldots,J,
\end{equation}
for any $M\in\mathcal{M}.$ We will say that a martingale $M$ is
\textit{weakly optimal}, or just \textit{optimal}, at $j,$ for some $j=0,...,J,$ if
\begin{equation}
Y_{j}^{\star} =\mathsf{E}_{\mathcal{F}_{j}}\left[  \max_{j\leq r\leq J}\left(
Z_{r}-M_{r}+M_{j}\right)  \right]  . \label{defopt}%
\end{equation}
The set of all martingales (weakly) optimal at $j$ will be denoted by $\mathcal{M}%
^{\circ,j}.$ The set of martingales optimal at $j$ for all $j=0,\ldots,J,$ is
denoted by $\mathcal{M}^{\circ}.$ We say that a martingale $M$ is
\textit{surely optimal} at $j,$ for some $j=0,...,J,$ if
\begin{equation}
Y_{j}^{\star}=\max_{j\leq r\leq J}\left(  Z_{r}-M_{r}+M_{j}\right)  \text{
\ \ almost surely.\ } \label{sure}%
\end{equation}
The set of all surely optimal martingales at $j$ will be denoted by
$\mathcal{M}^{\circ\circ,j}.$ The set of surely optimal martingales at $j$ for
all $j=0,\ldots,J,$ is denoted by $\mathcal{M}^{\circ\circ}.$ Note that,
obviously, $\mathcal{M}^{\circ\circ}$ $\subset$ $\mathcal{M}^{\circ}$
$\subset$ $\mathcal{M}.$

Now there always exists at least one surely optimal martingale, the so-called
Doob-martingale coming from the Doob decomposition of the Snell envelope
$(Y_{j}^{\star})_{j\geq0}.$ Indeed, consider the Doob decomposition of
$Y^{\star},$ that is,
\begin{equation}
Y_{j}^{\star}=Y_{0}^{\star}+M_{j}^{\star}-A_{j}^{\star}, \label{Doob0}%
\end{equation}
where $M^{\star}$ is a martingale with $M_{0}^{\star}=0,$ and $A^{\star}$ is
predictable with $A_{0}^{\star}=0.$ It follows immediately that
\begin{equation}
M_{j}^{\star}=\sum\limits_{l=1}^{j} (Y_{l}^{\star}-\mathsf{E}_{\mathcal{F}%
_{l-1}}[Y_{l}^{\star}]) ,\quad A_{j}^{\star}=\sum\limits_{l=1}^{j}
(Y_{l-1}^{\star}-\mathsf{E}_{\mathcal{F}_{l-1}}[Y_{l}^{\star}]),
\label{AddMart}%
\end{equation}
and so $A^{\star}$ is non-decreasing due to the fact that $Y^{\star}$ is a
supermartingale. One thus has by (\ref{Doob0}) on the one hand%
\[
\max_{j\leq r\leq J}(Z_{r}-M_{r}^{\star}+M_{j}^{\star})=Y_{j}^{\star}%
+\max_{j\leq r\leq J}(Z_{r}-Y_{r}^{\star}+A_{j}^{\star}-A_{r}^{\star})\leq
Y_{j}^{\star}%
\]
and due to (\ref{2up}) on the other hand%
\[
\mathsf{E}_{\mathcal{F}_{j}}\left[  \max_{j\leq r\leq J}(Z_{r}-M_{r}^{\star
}+M_{j}^{\star})\right]  \geq Y_{j}^{\star}.
\]
Thus, it follows that (\ref{sure}) holds for
arbitrary $j,$ hence $M^{\star}\in\mathcal{M}^{\circ\circ}.$
Furthermore we have the following properties of the sets \((\mathcal{M}^{\circ,j})\) and \((\mathcal{M}^{\circ
\circ,j}).\)
\begin{proposition}
\label{convex} The sets $\mathcal{M}^{\circ,j}$ and $\mathcal{M}^{\circ
\circ,j}$ for $j=0,...,J,$ $\mathcal{M}^{\circ},$ and $\mathcal{M}^{\circ
\circ}$ are convex.
\end{proposition}

As an immediate consequence of Proposition~\ref{convex}; if there exist more
than one weakly (respectively surely) optimal martingale, then there exist
infinitely many weakly (respectively surely) optimal martingales.

\begin{proposition}
\label{prop:propM} It holds that $M$ $\in\mathcal{M}^{\circ,j}$ for some
$0\leq j\leq J,$ if and only if for any optimal stopping time $\tau_{j}%
^{\star}\geq j$ satisfying%
\[
Y_{j}^{\star}=\sup_{\tau\geq j}\mathsf{E}_{\mathcal{F}_{j}}[  Z_{\tau
}]  =\mathsf{E}_{\mathcal{F}_{j}}[  Z_{\tau_{j}^{\star}}]  ,
\]
one has that%
\[
\max_{j\leq r\leq J}\left(  Z_{r}-M_{r}\right)  =Z_{\tau_{j}^{\star}}%
-M_{\tau_{j}^{\star}}.%
\]

\end{proposition}

\begin{proof}
Let $\tau_{j}^{\star}\geq j$ be an optimal stopping time. Suppose that $M$
$\in\mathcal{M}^{\circ,j}.$ On the one hand, one trivially has%
\[
\max_{j\leq r\leq J}\left(  Z_{r}-M_{r}\right)  -\left(  Z_{\tau_{j}^{\star}%
}-M_{\tau_{j}^{\star}}\right)  \geq0
\]
and on the other, since $M$ $\in\mathcal{M}^{\circ,j}$ (see (\ref{defopt})),%
\[
\mathsf{E}_{\mathcal{F}_{j}}\left[  \max_{j\leq r\leq J}\left(  Z_{r}%
-M_{r}\right)  -\left(  Z_{\tau_{j}^{\star}}-M_{\tau_{j}^{\star}}\right)
\right]  =Y_{j}^{\star}-M_{j}-\left(  Y_{j}^{\star}-M_{j}\right)  =0,\text{
\ \ hence}%
\]%
\begin{equation}
\max_{j\leq r\leq J}\left(  Z_{r}-M_{r}\right)  =Z_{\tau_{j}^{\star}}%
-M_{\tau_{j}^{\star}}\ \ \text{almost surely.} \label{co}%
\end{equation}
The converse follows from (\ref{co}) by taking conditional $\mathcal{F}_{j}$-expectations.
\end{proof}

It will be shown below that the class of the optimal martingales $\mathcal{M}^{\circ}$
may be considerably large. In fact, any such martingale can be seen as a
perturbation of the Doob martingale $(M_{j}^{\star}).$ For this, let us
introduce some further notation and define $\tau^{0}:=0^{-}$ with $0^{-}<0$ by
convention and let, for $l \geq1,$ $\tau^{l}$ be the first optimal stopping
time strictly after $\tau^{l-1}.$ That is, if $\tau^{l-1}<J,$ we define
recursively%
\[
\tau^{l}=\inf\left\{  \tau^{l-1}<i\leq J: Z_{i}\geq\mathsf{E}_{\mathcal{F}%
_{i}}\left[  Y_{i+1}^{\star}\right]  \right\}  ,
\]
where $Y_{J+1}^{\star}:=0.$ There so will be a last number, $l_{J}$ say, with
$\tau^{l_{J}}=J.$ Further, the family $\left(  \tau_{i}^{\star}\right)
_{i\geq0}$ defined by%
\begin{equation}
\tau_{i}^{\star}=\tau^{l}\text{ \ \ for \ }\tau^{l-1}<i\leq\tau^{l},\text{
\ \ }l\geq1, \label{pol}%
\end{equation}
is a \emph{consistent optimal stopping family} in the sense that $Y_{j}%
^{\star}=\mathsf{E}_{\mathcal{F}_{j}}[ Z_{\tau_{j}^{\star}}]$ and that
$\tau_{i}^{\star}>i$ implies $\tau_{i}^{\star}$ $=$ $\tau_{i+1}^{\star}.$

The next lemma provides a corner stone for an explicit structural
characterization of (weakly) optimal martingales.

\begin{lemma}
\label{lemopt} $M\in\mathcal{M}^{\circ}$ if and only if $M$ is an adapted martingale
with $M_{0}=0$ such that the identities%
\begin{align*}
\text{(i) \ \ }\max_{\tau^{l-1}<r\leq\tau^{l}}(Z_{r}-M_{r})  &  =Z_{\tau^{l}%
}-M_{\tau^{l}}\text{ \ \ if \ \ }l\geq1,\\
\text{(ii) \ \ }\max_{\tau^{l-1}\leq r\leq\tau^{l}}(Z_{r}-M_{r})  &
=Z_{\tau^{l-1}}-M_{\tau^{l-1}}\text{ \ if \ \ }l>1
\end{align*}
hold.
\end{lemma}

The following lemma anticipates  sufficient conditions for a martingale \(M\) to be
optimal, that is, to be a member of $\mathcal{M}^{\circ}.$

\begin{lemma}
\label{charopt} Let $(\mathcal{S}_{i})_{0\leq i\leq J}$ be an adapted sequence
with $\mathcal{S}_{0}=0$ and consider the \textquotedblleft
shifted\textquotedblright\ Doob martingale
\[
M_{i}=M_{i}^{\star}-\mathcal{S}_{i},\text{ \ \ }0\leq i\leq J.
\]
Let $l_{i}\geq1$ be the unique number such that $\tau^{l_{i}-1}<i\leq
\tau^{l_{i}}$ for any $0\leq i\leq J.$ If $\mathcal{S}$ satisfies for all
$0\leq i\leq J,$
\begin{align}
\max_{\tau^{l_{i}-1}<r\leq i}\left(  Z_{r}-Y_{r}^{\star}+\mathcal{S}%
_{r}-\mathcal{S}_{i}\right)   &  \leq0 \label{c1}\\
Z_{\tau^{l_{i}-1}}-\mathsf{E}_{\mathcal{F}_{\tau^{l_{i}-1}}}\left[
Y_{\tau^{l_{i}-1}+1}^{\star}\right]  +\mathcal{S}_{\tau^{l_{i}-1}}%
-\mathcal{S}_{i}  &  \geq0,\label{c2}%
\end{align}
for $\tau^{l_{i}-1}<i\leq\tau^{l_{i}}$ and $l_{i}>1$,  then $M$ satisfies the identities (i)-(ii) in Lemma~\ref{lemopt}.
\end{lemma}

\begin{corollary}
\label{eqco} Let us represent an (arbitrary) adapted $\mathcal{S}$ with
$\mathcal{S}_{0}=0$ by%
\begin{equation}
\mathcal{S}_{i+1}=\mathcal{S}_{i}+\zeta_{i+1},\text{
\ \ }0\leq i<J, \label{ds}%
\end{equation}
where each $\zeta_{i+1}$ is a $\mathcal{F}_{i+1}$-measurable random
variable. Then the conditions (\ref{c1}) and (\ref{c2}) are equivalent to the
following ones.
\begin{description}
\item [(i)]  On the $\mathcal{F}_{i}$-measurable event $\left\{  \tau^{l_{i}-1}%
<i<\tau^{l_{i}}\right\}  $ it holds that%
\begin{align}
\zeta_{i+1}  &  \geq\max_{\tau^{l_{i}-1}<r\leq i}\left(  Z_{r}-Y_{r}^{\star
}+\mathcal{S}_{r}-\mathcal{S}_{i}\right)  \text{ \ \ and}\label{imp}\\
\zeta_{i+1}  &  \leq Z_{\tau^{l_{i}-1}}-\mathsf{E}_{\mathcal{F}_{\tau
^{l_{i}-1}}}\left[  Y_{\tau^{l_{i}-1}+1}^{\star}\right]  +\mathcal{S}%
_{\tau^{l_{i}-1}}-\mathcal{S}_{i}\text{ \ for \ }l_{i}>1; \label{imm}%
\end{align}
\item [(ii)] On $\left\{  \tau^{l_{i}}=i\right\}  $ one has that%
\begin{equation}
\zeta_{i+1}\leq Z_{i}-\mathsf{E}_{\mathcal{F}_{i}}\left[  Y_{i+1}^{\star
}\right]  . \label{pp}%
\end{equation}
\end{description}
\end{corollary}
\begin{proof}
Indeed, take $j$ such that $\left\{  \tau^{l_{j}-1}<j\leq\tau^{l_{j}}\right\}
,$ $l_{j}\geq1.$ If $j-1>\tau^{l_{j}-1}$ then $l_{j-1}=l_{j}$ and (\ref{imp})
and (\ref{imm}) imply with $i=j-1$ via (\ref{ds}),%
\begin{align*}
0  &  \geq\max_{\tau^{l_{j}-1}<r\leq j-1}\left(  Z_{r}-Y_{r}^{\star
}+\mathcal{S}_{r}-\mathcal{S}_{j}\right)  \text{ \ \ and}\\
0  &  \leq Z_{\tau^{l_{j}-1}}-\mathsf{E}_{\mathcal{F}_{\tau^{l_{j}-1}}}\left[
Y_{\tau^{l_{j}-1}+1}^{\star}\right]  +\mathcal{S}_{\tau^{l_{j}-1}}%
-\mathcal{S}_{j}\text{ \ for \ }l_{j}>1,\text{\ }%
\end{align*}
respectively, which in turn imply (\ref{c1}) (note that $Z_{j}-Y_{j}^{\star
}\leq0$) and (\ref{c2}), respectively. Further if $j-1\ngtr\tau^{l_{j}-1}$ we
have to distinguish between $j=0\wedge l_{0}=1$ and $j=\tau^{l_{j}-1}+1\wedge
l_{j}>1.$ In both cases (\ref{c1}) is trivially fulfilled, while (\ref{c2}) is
void in the first case, and in the second case it reads,%
\[
0\leq Z_{\tau^{l_{j}-1}}-\mathsf{E}_{\mathcal{F}_{\tau^{l_{j}-1}}}[
Y_{\tau^{l_{j}-1}+1}^{\star}]  +\mathcal{S}_{\tau^{l_{j}-1}}%
-\mathcal{S}_{\tau^{l_{j}-1}+1},\text{ \ \ \ }l_{j}>1,
\]
which is implied by (\ref{ds}) and (\ref{pp}) for $i=j-1=\tau^{l_{i}}%
=\tau^{l_{j-1}}=\tau^{l_{j}-1}.$ The converse direction, that is from
(\ref{c1}) and (\ref{c2}) to (\ref{imp}), (\ref{imm}), (\ref{pp}), goes
similarly and is left to the reader.
\end{proof}
\begin{corollary}
\label{ifp} By Corollary~\ref{eqco} there always exists an adapted process $\mathcal{S}$
satisfying (\ref{imp}), (\ref{imm}), (\ref{pp}) with
$\mathsf{E}_{i}\left[  \zeta_{i+1}\right]  =0$ for $0\leq i<J$ due to
(\ref{c1}) and (\ref{c2}). Hence, there exist martingales $\mathcal{S}$ that
satisfy Lemma~\ref{charopt}. By Lemma~\ref{lemopt}, for any such martingale
$\mathcal{S},$ $M=M^{\star}-\mathcal{S}\in\mathcal{M}^{\circ},$ that is, $M$ is the
optimal martingale.
\end{corollary}

Interestingly, the converse to Corollary~\ref{ifp} is also true and we so have
the following characterization theorem.

\begin{theorem}
\label{cor:main} It holds that $M\in\mathcal{M}^{\circ}$ if and only if
$M=M^{\star}-\mathcal{S},$ where $\mathcal{S}$ is a martingale with
$\mathcal{S}_{0}=0$ that satisfies (\ref{c1}) and (\ref{c2}) in
Lemma~\ref{charopt}.
\end{theorem}
The proofs of Lemmas~\ref{lemopt}-\ref{charopt} and Theorem~\ref{cor:main} are
given in Section~\ref{SecProofs}.
In fact, Theorem~\ref{cor:main} reveals that, besides the Doob martingale,
there generally exists a large set of optimal martingales $M\in\mathcal{M}%
^{\circ}.$ From Theorem~\ref{cor:main} we also obtain a characterization of
the \emph{surely} optimal martingales which is essentially the older result in
\cite{J_SchZhaHua}, Thm.~6 (see Section~\ref{SecProofs} for the proof).

\begin{corollary}
\label{alms} It holds that $M\in\mathcal{M}^{\circ\circ}$ if and only if
$M=M^{\star}-\mathcal{S}$ with $\mathcal{S}$ represented by (\ref{ds}) with
all $\mathsf{E}_{\mathcal{F}_{i}}\left[  \zeta_{i+1}\right]  =0,$ $\zeta
_{i+1}$ satisfying (\ref{pp}) for $i=\tau^{l_{i}}$, and $\zeta_{i+1}=0$ for
$\tau^{l_{i}-1}<i<\tau^{l_{i}},$ $l_{i}\geq1.$
\end{corollary}

In applications of dual optimal stopping, hence dual martingale minimization,
it is usually enough to find martingales $M$ that are ``close to'' surely
optimal ones, merely at some specific point in time $i$, that is, $M$ $\in$
$\mathcal{M}^{\circ\circ,i}$. Naturally, since $\mathcal{M}^{\circ,i}$
$\supset$ $\mathcal{M}^{\circ},$ we may expect that in general the family of
undesirable (not surely) optimal martingales at a specific time may be even
much larger than the family $\mathcal{M}^{\circ}$ characterized by
Theorem~\ref{cor:main}. A characterization of $\mathcal{M}^{\circ,i}$ and
$\mathcal{M}^{\circ\circ,i}$ is given by the next theorem, where we take $i=0$
without loss of generality. The proof is given in Section~\ref{SecProofs}.

\begin{theorem}
\label{i0}
The following statements hold.
\begin{description}
\item [(i)] $M=M^{\star}-\mathcal{S}\in\mathcal{M}^{\circ,0}$
for some martingale $\mathcal{S}$ represented by (\ref{ds}), if and only if%
\begin{align}
\max_{0\leq r<j}\left(  Z_{r}-Y_{r}^{\ast}-\mathcal{S}_{j}+\mathcal{S}%
_{r}\right)   &  \leq0\text{ \ \ for \ \ }0\leq j\leq\tau^{\star}\text{
\ \ and}\label{dd1}\\
\mathcal{S}_{j}-\mathcal{S}_{\tau^{\star}}  &  \leq Y_{j}^{\star}-Z_{j}%
+A_{j}^{\ast}\text{ \ \ for \ \ }\tau^{\star}<j\leq J, \label{dd2}%
\end{align}
where $A_{j}^{\ast}=0$ (see (\ref{Doob0})) for all $0\leq
j\leq\tau^{\star}.$
\item [(ii)] $M=M^{\star}-\mathcal{S}\in\mathcal{M}^{\circ\circ,0},$ if and only if
\begin{align}
\mathcal{S}_{j}  &  =0\ \ \text{\ for \ \ }0\leq j\leq\tau^{\star}, \label{df1}\\
\mathcal{S}_{j}  &  \leq Y_{j}^{\star}-Z_{j}+A_{j}^{\ast}\text{ \ \ for
\ \ }\tau^{\star}<j\leq J. \label{df2}%
\end{align}
\end{description}
\end{theorem}

After dropping the nonnegative term $Y_{j}^{\star}-Z_{j}$ in the
right-hand-sides of (\ref{dd2}) and (\ref{df2}) we may obtain tractable
sufficient conditions for a martingale to be optimal or surely optimal at a
single date, respectively. In the spirit of Corollary~\ref{eqco} they may be
formulated in the following way.

\begin{corollary}
Let $M=M^{\star}-\mathcal{S}$ for some martingale $\mathcal{S}$ represented by
(\ref{ds}), then
\begin{description}
\item [(i)] $M\in\mathcal{M}^{\circ,0}$ if%
\begin{align}
\zeta_{j}\text{ }  &  \geq\max_{0\leq r<j}\left(  Z_{r}-Y_{r}^{\ast
}-\mathcal{S}_{j-1}+\mathcal{S}_{r}\right)  \text{\ \ \ for \ \ }1\leq
j\leq\tau^{\star}\text{ \ \ and}\nonumber\\
\zeta_{j}  &  \leq A_{j}^{\ast}+\mathcal{S}_{\tau^{\star}}-\mathcal{S}%
_{j-1}\text{ \ \ for \ \ }\tau^{\star}<j\leq J, \label{rhs}%
\end{align}
\item [(ii)] $M\in\mathcal{M}^{\circ\circ,0}$ if $\zeta_{j}=0$ for $0\leq j\leq
\tau^{\star},$ and
\begin{equation}
\zeta_{j}\leq A_{j}^{\ast}-\mathcal{S}_{j-1}\text{ \ \ for \ \ }\tau^{\star
}<j\leq J. \label{rhs1}%
\end{equation}
In particular, the right-hand-sides in (\ref{rhs}) and (\ref{rhs1}) are
$\mathcal{F}_{j-1}$-measurable.
\end{description}
\end{corollary}

\begin{remark}
While the class of optimal martingales $\mathcal{M}^{\circ,0}$ may be quite
large in general, it is still possible that it is just a singleton (containing
the Doob martingale only). For example, let the cash-flow $Z\geq0 $ be a
martingale itself, then it is easy to see that the only optimal martingale (at
$0$) is $M=M^{\star}=Z-Z_{0}$ (the proof is left as an easy exercise).
\end{remark}

\section{Randomized dual martingale representations}

\label{SecRan}

Let $(\Omega_{0},\mathcal{B})$ be some auxiliary measurable space that is
``rich enough''. Let us consider random variables on $\widetilde{\Omega
}:=\Omega\times\Omega_{0}$ that are measurable with respect to the $\sigma
$-field $\widetilde{\mathcal{F}}:=\sigma\left\{  F\times B:F\in\mathcal{F}%
,\text{ }B\in\mathcal{B}\right\}  .$ While abusing notation a bit,
$\mathcal{F} $ and $\mathcal{F}_{j}$ are identified with $\sigma\left\{
F\times\Omega_{0}:F\in\mathcal{F}\right\}  \subset\widetilde{\mathcal{F}}$ and
$\sigma\left\{  F\times\Omega_{0}:F\in\mathcal{F}_{j}\right\}  \subset
\widetilde{\mathcal{F}},$ respectively. Let further $\mathsf{P}$ be the given
\textquotedblleft primary\textquotedblright\ measure on $(\Omega
,\mathcal{F}),$ and $\widetilde{\mathsf{P}}$ be an extension of $\mathsf{P}$
to $(\widetilde{\Omega},\widetilde{\mathcal{F}})$ in the sense that%
\[
\widetilde{\mathsf{P}}\left(  \Omega_{0}\times F\right)  =\mathsf{P}\left(
F\right)  \text{ \ \ for all \ \ }F\in\mathcal{F}.
\]
In particular, if $X:\widetilde{\Omega}\rightarrow\mathbb{R}$ is $\mathcal{F}
$-measurable, then $\left\{  \left(  \omega,\omega_{0}\right)  :X\left(
\omega,\omega_{0}\right)  \leq x\right\}=$
\,
$ \left\{  \left(  \omega,\omega
_{0}\right)  :\omega\in F_{x}\right\}  $ for some $F_{x}\in\mathcal{F}$, that
is, $X$ does not depend on $\omega_{0}.$
We now introduce randomized or ``pseudo''
 martingales as random perturbations of $\mathcal{F}$-adapted martingales of
the form (\ref{ds}). Let $(\eta_{j})_{j\geq0}$ be random variables on $(
\widetilde{\Omega},\widetilde{\mathcal{F}},\widetilde{\mathsf{P}}) $ such that
$\widetilde{\mathsf{E}}_{\mathcal{F}}\left[  \eta_{j}\right]  =0$ for
$j=0,\ldots,J.$ Then
\begin{equation}
\widetilde{M}_{j}:=M_{j}-\eta_{j}=M_{j}^{\star}-\mathcal{S}_{j}-\eta_{j}
\label{ran0}%
\end{equation}
is said to be a pseudo martingale. As such, $\widetilde{M}$ is not an
$\mathcal{F}$-martingale but $\widetilde{\mathsf{E}}_{\mathcal{F}}%
[\widetilde{M}]$ is.
The results below on pseudo-martingales provide the key motivation for
randomized dual optimal stopping. All proofs in this section are deferred to
Section~\ref{SecProofs}.

\begin{proposition}
\label{ASP} For any $\widetilde{M}$ of the form (\ref{ran0}) one has the upper
estimate%
\begin{equation}
\widetilde{\mathsf{E}}\Bigl[ \max_{0\leq j\leq J}( Z_{j}-\widetilde{M}_{j})
\Bigr] \geq Y_{0}^{\star}. \label{uppb}%
\end{equation}
If $\mathcal{S}=0,$ that is,
\begin{equation}
\widetilde{M}_{j}=M_{j}^{\star}-\eta_{j} \label{ran}%
\end{equation}
and the random perturbations $\left(  \eta_{j}\right)  $ satisfy in addition
\begin{equation}
\eta_{j}\leq Y_{j}^{\star}-Z_{j}+A_{j}^{\star},\quad\widetilde{\mathsf{P}%
}-a.s.\quad j=0,\ldots,J, \label{asl}%
\end{equation}
with $(A_{j}^{\star})$ defined in \eqref{Doob0}, then one has the almost sure
identity%
\begin{equation}
Y_{0}^{\star}=\max_{0\leq j\leq J}( Z_{j}-\widetilde{M}_{j}) \quad
\widetilde{\mathsf{P}}\text{-a.s.} \label{as}%
\end{equation}
Moreover, for the first optimal stopping time $\tau^{\star}$ $:=$ $\tau
_{0}^{\star}$ (see (\ref{pol})) one must have that $\eta_{\tau^{\star}}=0 $
a.s., and if $\tau^{\star}$ is strict in the sense that
\[
Y_{\tau^{\star}}^{\star}-\mathsf{E}_{\mathcal{F}_{\tau^{\star}}}\left[
Y_{\tau^{\star}+1}^{\star}\right]  >0,
\]
then $j=\tau^{\star}$ is the only time $j$ where  $\eta_{j}=0.$
\end{proposition}

Due to the following theorem, any (weakly or surely) optimal non Doob martingale 
turns to a non optimal one in the sense that
\begin{equation}
\widetilde{\mathsf{E}}\Bigl[\max_{0\leq j\leq J}(Z_{j}-\widetilde{M}%
_{j})\Bigr]>Y_{0}^{\star} \label{asw}%
\end{equation}
after a particular \textquotedblleft optimal\textquotedblright\ randomization.

\begin{theorem}
\label{opran} Suppose that $M\in\mathcal{M}^{\circ,0}$ and let $(\eta_{j})$ be
a sequence of random variables as in Proposition~\ref{ASP}, given by%
\begin{equation}
\eta_{j}=\xi_{j}\left(  Y_{j}^{\star}-Z_{j}+A_{j}^{\star}\right)  ,\text{
\ \ }0\leq j\leq J, \label{etu}%
\end{equation}
where the $(\xi_{j})$ are assumed to be i.i.d. distributed on $(-\infty,1],$
independent of $\mathcal{F}$ with $\widetilde{\mathsf{E}}\left[  \xi
_{j}\right]  =0.$ It is further assumed that the r.v. $(\xi_{j})$ have a joint continuous
density $p$ supported on $(-\infty,1]$ with $p(1)>0$. As such the
randomizers (\ref{etu}) satisfy (\ref{asl}), and Proposition~\ref{ASP} thus
provides an upper bound (\ref{uppb}) due to the pseudo martingale
$\widetilde{M}=M-\eta.$ Now, for the randomized martingale $\widetilde{M}$ one has (\ref{asw}) if $M\ne M^{\star}$
with positive probability.
\end{theorem}

The following corollary states that an optimally randomized non Doob martingale in $\mathcal{M}^{\circ,0},$ which 
is thus  suboptimal in the sense of (\ref{asw}) due to the previous theorem, cannot have zero variance. The proof relies on 
Theorem~\ref{opran}.

\begin{corollary}
\label{pseudovar} Let $M\in\mathcal{M}^{\circ,0},$ $(\eta_{j})$ as in
Theorem~\ref{opran}, and $\widetilde{M}=M-\eta$. Then $\mathrm{Var}\bigl(
\max_{0\leq j\leq J}(Z_{j}-\widetilde{M}_{j})\bigr)  $ $=$ $0$ if and only if
$M=M^{\star}.$
\end{corollary}

\subsubsection*{Discussion}

Proposition~\ref{ASP} provides us with a remarkable freedom of perturbing the
Doob martingale randomly while (\ref{as}) remains true. The bottom line of
Theorem~\ref{opran} is that randomization under condition (\ref{asl}) of an
optimal, or even surely optimal, but \textit{non}-Doob martingale results in a
non optimal (pseudo) martingale, while any randomization of the Doob
martingale under (\ref{asl}) remains a surely optimal pseudo martingale. This
is an important feature, since in this way martingale candidates that are
optimal but not equal to the (surely optimal) Doob martingale can be sorted
out by randomization.

\section{Proofs}

\label{SecProofs}

\subsection{Proof of Lemma~\ref{convex}}

It is enough to show the convexity of $\mathcal{M}^{\circ,j}$ and
$\mathcal{M}^{\circ\circ,j}$ for any $j.$ For any $M,M^{\prime}\in
\mathcal{M}^{\circ,j}$ and $\theta\in(0,1)$ one has%
\begin{align*}
&  \mathsf{E}_{\mathcal{F}_{j}}\left[  \max_{j\leq r\leq J}\left(
Z_{r}-\left(  \theta M_{r}+(1-\theta)M_{r}^{\prime}\right)  \right)  +\theta
M_{j}+(1-\theta)M_{j}^{\prime}\right] \\
&  =\mathsf{E}\left[  \max_{j\leq r\leq J}\left(  \theta\left(  Z_{r}%
-M_{r}+M_{j}\right)  +(1-\theta)\left(  Z_{r}-M_{r}^{\prime}+M_{j}^{\prime
}\right)  \right)  \right] \\
&  \leq\theta\mathsf{E}\left[  \max_{j\leq r\leq J}\left(  Z_{r}-M_{r}%
+M_{j}\right)  \right]  +(1-\theta)\mathsf{E}\left[  \max_{j\leq r\leq
J}\left(  Z_{r}-M_{r}^{\prime}+M_{j}^{\prime}\right)  \right]  =Y_{j}^{\star}%
\end{align*}
while by (\ref{2up}),%
\[
\mathsf{E}_{\mathcal{F}_{j}}\left[  \max_{j\leq r\leq J}\left(  Z_{r}-\left(
\theta M_{r}+(1-\theta)M_{r}^{\prime}\right)  +\theta M_{j}+(1-\theta
)M_{j}^{\prime}\right)  \right]  \geq Y_{j}^{\star}.
\]
Similarly, for any $M,M^{\prime}\in\mathcal{M}^{\circ\circ,j}$ and $\theta
\in(0,1)$ we have%
\begin{align*}
&  \max_{j\leq r\leq J}\left(  Z_{r}-\left(  \theta M_{r}+(1-\theta
)M_{r}^{\prime}+\theta M_{j}+(1-\theta)M_{j}^{\prime}\right)  \right) \\
&  =\max_{j\leq r\leq J}\left(  \theta\left(  Z_{r}-M_{r}+M_{j}\right)
+(1-\theta)\left(  Z_{r}-M_{r}^{\prime}+M_{j}^{\prime}\right)  \right) \\
&  \leq\theta\max_{j\leq r\leq J}\left(  Z_{r}-M_{r}+M_{j}\right)
+(1-\theta)\max_{0\leq r\leq J}\left(  Z_{r}-M_{r}^{\prime}+M_{j}^{\prime
}\right)  =Y_{j}^{\star}%
\end{align*}
while by (\ref{2up}),%
\[
\mathsf{E}_{\mathcal{F}_{j}}\left[  \max_{j\leq r\leq J}\left(  Z_{r}-\left(
\theta M_{r}+(1-\theta)M_{r}^{\prime}\right)  +\theta M_{j}+(1-\theta
)M_{j}^{\prime}\right)  \right]  \geq Y_{j}^{\star}.
\]
In both cases the sandwich property completes.

\subsection{Proof of Lemma~\ref{lemopt}}

Suppose that $M$ is a martingale with $M_{0}=0$ such that Lemma~\ref{lemopt}%
-(i) and (ii) hold. Then (ii) implies for $q\geq1$ that%
\begin{equation}
Z_{\tau^{1}}-M_{\tau^{1}}\geq Z_{\tau^{2}}-M_{\tau^{2}}\geq...\geq Z_{\tau
^{q}}-M_{\tau^{q}} \label{ket}%
\end{equation}
Now take $0\leq i\leq J$ arbitrarily, and let $q_{i}\geq1$ be such that
$\tau^{q_{i}-1}<i\leq\tau^{q_{i}}$ (Note that $q_{i}$ is unique and
$\mathcal{F}_{i}$ measurable). Then due to Lemma~\ref{lemopt}-(i) and
(\ref{ket}),
\begin{align*}
\max_{i\leq r\leq J}(Z_{r}-M_{r})  &  =\max\left(  \max_{i\leq r\leq
\tau^{q_{i}}}(Z_{r}-M_{r}),\max_{q>q_{i}}\max_{\tau^{q-1}<r\leq\tau^{q}}%
(Z_{r}-M_{r})\right) \\
&  =\max\left(  Z_{\tau^{q_{i}}}-M_{\tau^{q_{i}}},\max_{q>q_{i}}\left(
Z_{\tau^{q}}-M_{\tau^{q}}\right)  \right) \\
&  =\max\left(  Z_{\tau^{q_{i}}}-M_{\tau^{q_{i}}},Z_{\tau^{q_{i}+1}}%
-M_{\tau^{q_{i}+1}}\right)  =Z_{\tau^{q_{i}}}-M_{\tau^{q_{i}}}.
\end{align*}
On the other hand, one has $\tau_{i}^{\star}=\tau^{q_{i}}$ (see (\ref{pol})).
Thus, by Proposition~\ref{prop:propM}, $M\in\mathcal{M}^{\circ,i}$ and hence
$M\in\mathcal{M}^{\circ}$ since $i$ was taken arbitrarily.

Conversely, suppose that $M\in\mathcal{M}^{\circ}.$ So for any $0\leq i\leq
J,$%
\[
\max_{i\leq r\leq J}(Z_{r}-M_{r})=Z_{\tau_{i}^{\star}}-M_{\tau_{i}^{\star}}%
\]
by Proposition~\ref{prop:propM}. For $l=1$ one thus has%
\[
\max_{\tau^{0}<r\leq J}(Z_{r}-M_{r})=\max_{0\leq r\leq J}(Z_{r}-M_{r}%
)=Z_{\tau_{0}^{\star}}-M_{\tau_{0}^{\star}}=Z_{\tau^{1}}-M_{\tau^{1}}%
\]
and for $l>1$ it holds that%
\begin{align*}
\max_{\tau^{l-1}<r\leq J}(Z_{r}-M_{r})  &  =\sum_{k=0}^{J-1}1_{\{\tau
^{l-1}=k\}}\max_{k+1\leq r\leq J}(Z_{r}-M_{r}) \\
& =\sum_{k=0}^{J-1}1_{\{\tau
^{l-1}=k\}}\left(  Z_{\tau_{k+1}^{\star}}-M_{\tau_{k+1}^{\star}}\right) \\
&  =\sum_{k=0}^{J-1}1_{\{\tau^{l-1}=k\}}\left(  Z_{\tau^{l}}-M_{\tau^{l}%
}\right)  =Z_{\tau^{l}}-M_{\tau^{l}}.
\end{align*}
That is, (i) is shown. Next, for any $l>1$ it holds%
\begin{align*}
\max_{\tau^{l-1}\leq r\leq J}(Z_{r}-M_{r})  &  =\sum_{k=0}^{L}1_{\{\tau
^{l-1}=k\}}\max_{k\leq r\leq J}(Z_{r}-M_{r})\\
&  =\sum_{k=0}^{L}1_{\{\tau^{l-1}=k\}}\left(  Z_{\tau_{k}^{\star}}-M_{\tau
_{k}^{\star}}\right) \\
&  =Z_{\tau_{\tau^{l-1}}^{\star}}-M_{\tau_{\tau^{l-1}}^{\star}}=Z_{\tau^{l-1}%
}-M_{\tau^{l-1}}%
\end{align*}
which implies (ii).


\subsection{Proof of Lemma~\ref{charopt}}

Assume that $\mathcal{S}$ is adapted with $\mathcal{S}_{0}=0$ and that
$\mathcal{S}$ satisfies (\ref{c1}) and (\ref{c2}). For $l>1$ and $\tau
^{l-1}<r\leq\tau^{l}$ we may write,
\begin{align}
Z_{r}-M_{r}  &  =Z_{r}-M_{r}^{\star}+\mathcal{S}_{r}\label{lm1}\\
&  =Z_{r}-M_{\tau^{l-1}}^{\star}+M_{\tau^{l-1}}^{\star}-M_{r}^{\star
}+\mathcal{S}_{r}\nonumber\\
&  =Z_{r}-M_{\tau^{l-1}}^{\star}+\mathcal{S}_{r}-\sum_{k=\tau^{l-1}+1}%
^{r}\left(  Y_{k}^{\star}-\mathsf{E}_{\mathcal{F}_{k-1}}\left[  Y_{k}^{\star
}\right]  \right) \nonumber\\
&  =Z_{r}-M_{\tau^{l-1}}^{\star}+\mathcal{S}_{r}\nonumber\\
&  -\sum_{k=\tau^{l-1}+1}^{r}Y_{k}^{\star}+\sum_{k=\tau^{l-1}+1}%
^{r-1}\mathsf{E}_{\mathcal{F}_{k}}\left[  Y_{k+1}^{\star}\right]
+\mathsf{E}_{\mathcal{F}_{\tau^{l-1}}}\left[  Y_{\tau^{l-1}+1}^{\star}\right]
\nonumber\\
&  =Z_{r}-Y_{r}^{\star}-M_{\tau^{l-1}}^{\star}+\mathsf{E}_{\mathcal{F}%
_{\tau^{l-1}}}\left[  Y_{\tau^{l-1}+1}^{\star}\right]  +\mathcal{S}%
_{r}.\nonumber
\end{align}
By taking $r=\tau^{l}$ in (\ref{lm1}) and using $Z_{\tau^{l}}=Y_{\tau^{l}%
}^{\star}$ we then get%
\[
Z_{\tau^{l}}-M_{\tau^{l}}=-M_{\tau^{l-1}}^{\star}+\mathsf{E}_{\mathcal{F}%
_{\tau^{l-1}}}\left[  Y_{\tau^{l-1}+1}^{\star}\right]  +\mathcal{S}_{\tau^{l}}%
\]
and thus%
\[
Z_{r}-M_{r}=Z_{\tau^{l}}-M_{\tau^{l}}+Z_{r}-Y_{r}^{\star}+\mathcal{S}%
_{r}-\mathcal{S}_{\tau^{l}},\text{ \ \ }\tau^{l-1}<r\leq\tau^{l}.
\]
So from (\ref{c1}) we obtain with $i=\tau^{l},$ $l_{i}-1=l-1,$%
\[
Z_{r}-M_{r}\leq Z_{\tau^{l}}-M_{\tau^{l}}\text{ \ \ for \ \ }\tau^{l-1}%
<r\leq\tau^{l},
\]
i.e. Lemma~\ref{lemopt}-(i) for $l>1.$ If $l=1$ and $\tau^{1}=0,$
Lemma~\ref{lemopt}-(i) is trivially fulfilled. So let us consider $l=1$ and
$\tau^{1}>0.$ Analogously, we then may write for $\tau^{0}=0^{-}<0<r\leq
\tau^{1},$%
\begin{align}
Z_{r}-M_{r}  &  =Z_{r}-M_{r}^{\star}+\mathcal{S}_{r}=Z_{r}+\mathcal{S}%
_{r}-\sum_{k=1}^{r}\left(  Y_{k}^{\star}-\mathsf{E}_{\mathcal{F}_{k-1}}\left[
Y_{k}^{\star}\right]  \right) \nonumber\\
=Z_{r}+\mathcal{S}_{r}  &  -\sum_{k=1}^{r}Y_{k}^{\star}+\sum_{k=1}%
^{r-1}\mathsf{E}_{\mathcal{F}_{k}}\left[  Y_{k+1}^{\star}\right]
+\mathsf{E}_{\mathcal{F}_{\tau^{l-1}}}\left[  Y_{\tau^{l-1}+1}^{\star}\right]
\nonumber\\
&  =Z_{r}-Y_{r}^{\star}+\mathsf{E}_{\mathcal{F}_{0}}\left[  Y_{1}^{\star
}\right]  +\mathcal{S}_{r}. \label{im2}%
\end{align}
It is easy to see that (\ref{im2}) is also valid for $r=0,$ due to our
assumption $\tau^{1}>0.$ Thus, for $l=1$ and taking $r=\tau^{1}>0,$ we get
from (\ref{im2}),
\[
Z_{\tau^{1}}-M_{\tau^{1}}=\mathsf{E}_{\mathcal{F}_{0}}\left[  Y_{1}^{\star
}\right]  +\mathcal{S}_{\tau^{1}},
\]
whence (\ref{im2}) implies for $\tau^{0}=0^{-}<r\leq\tau^{1}$
\[
Z_{r}-M_{r}=Z_{r}-Y_{r}^{\star}+Z_{\tau^{1}}-M_{\tau^{1}}\leq Z_{\tau^{1}%
}-M_{\tau^{1}},
\]
that is Lemma~\ref{lemopt}-(i) holds also for $l=1.$

Let us now consider (ii) and take $l>1.$ Now for $\tau^{l-1}<r\leq\tau^{l}$
(\ref{lm1}) implies with $M_{\tau^{l-1}}^{\star}=$ $\mathcal{S}_{\tau^{l-1}%
}+M_{\tau^{l-1}},$%
\begin{equation}
Z_{r}-M_{r}=Z_{\tau^{l-1}}-M_{\tau^{l-1}}+Z_{r}-Y_{r}^{\star}+\mathsf{E}%
_{\mathcal{F}_{\tau^{l-1}}}\left[  Y_{\tau^{l-1}+1}^{\star}\right]
-Z_{\tau^{l-1}}+\mathcal{S}_{r}-\mathcal{S}_{\tau^{l-1}}. \label{tm}%
\end{equation}
Hence, since always $Z_{r}\leq Y_{r}^{\star},$ (\ref{c2}) implies for
$\tau^{l-1}<r\leq\tau^{l},$%
\begin{equation}
Z_{r}-M_{r}\leq Z_{\tau^{l-1}}-M_{\tau^{l-1}},\text{ \ \ }\tau^{l-1}<r\leq
\tau^{l}, \label{in00}%
\end{equation}
i.e. Lemma~\ref{lemopt}-(ii) is proved.


\subsection{Proof of Theorem~\ref{cor:main}}

If $M=M^{\star}-\mathcal{S}$, where $\mathcal{S}$ is a martingale with
$\mathcal{S}_{0}=0$ that satisfies (\ref{c1}) and (\ref{c2}) in
Lemma~\ref{charopt} then $M\in\mathcal{M}^{\circ}$ due to Corollary~\ref{ifp}.

Let us now consider the converse and assume that $M=M^{\star}-\mathcal{S}%
\in\mathcal{M}^{\circ}$ with $M_{0}=\mathcal{S}_{0}=0.$ Then $\mathcal{S}$ is
adapted and may be written in the form (\ref{ds}) where the $\zeta_{i+1}$ are
$\mathcal{F}_{i+1}$-measurable and $\mathsf{E}_{\mathcal{F}_{i}}\left[
\zeta_{i+1}\right]  =0$ for $0\leq$ $i<J.$ Since $M\in\mathcal{M}^{\circ}$
Lemma~\ref{lemopt}-(i) implies that for $l\geq1,$%
\begin{align}
\max_{\tau^{l-1}<r\leq\tau^{l}}(Z_{r}-Z_{\tau^{l}}+M_{\tau^{l}}^{\star}%
-M_{r}^{\star}+\mathcal{S}_{r}-\mathcal{S}_{\tau^{l}})  &  =0,\text{
\ \ hence}\nonumber\\
\max_{\tau^{l-1}<r\leq\tau^{l}}(Z_{r}-Y_{r}^{\ast}+\mathcal{S}_{r}%
-\mathcal{S}_{\tau^{l}})  &  =0 \label{h1}%
\end{align}
since for each $r$ with $\tau^{l-1}<r\leq\tau^{l}$ one has $Z_{\tau^{l}%
}-M_{\tau^{l}}^{\star}+M_{r}^{\star}=Z_{\tau_{r}^{\star}}-M_{\tau_{r}^{\star}%
}^{\star}+M_{r}^{\star}=Y_{r}^{\ast}$ because $M^{\star}\in\mathcal{M}%
^{\circ\circ}.$ We now show for any $i$ with $\tau^{l-1}<i\leq\tau^{l}$ that
(\ref{c1}) holds with $l_{i}=l$ by backward induction. For $i=\tau^{l_{i}}$ it
follows from (\ref{h1}). Now suppose that for some $i$ with $\tau^{l_{i}%
-1}<i<i+1\leq\tau^{l_{i}}$ it holds that%
\begin{equation}
1_{\left\{  \tau^{l_{i+1}-1}<i+1\leq\tau^{l_{i+1}}\right\}  }\max
_{\tau^{l_{i+1}-1}<r\leq i+1}(Z_{r}-Y_{r}^{\ast}+\mathcal{S}_{r}%
-\mathcal{S}_{i+1})\leq0. \label{hin}%
\end{equation}
One has by construction%

\[
\max_{\tau^{l_{i}-1}<r\leq i}(Z_{r}-Y_{r}^{\ast}+\mathcal{S}_{r}%
-\mathcal{S}_{i})=\zeta_{i+1}+\max_{\tau^{l_{i}-1}<r\leq i}(Z_{r}-Y_{r}^{\ast
}+\mathcal{S}_{r}-\mathcal{S}_{i+1}).
\]
Hence, since $\left\{  \tau^{l_{i}-1}<i<\tau^{l_{i}}\right\}  =\left\{
\tau^{l_{i}-1}<i\right\}  \cap\left\{  \tau^{l_{i}-1}<i+1\leq\tau^{l_{i}%
}\right\}  $ with $\left\{  \tau^{l_{i}-1}<i\right\}  \in\mathcal{F}_{i}$ and
$\left\{  \tau^{l_{i}-1}<i+1\leq\tau^{l}\right\}  \in\mathcal{F}_{i}$ (!),
$\mathsf{E}_{\mathcal{F}_{i}}\left[  \zeta_{i+1}\right]  =0,$ $l_{i}=l_{i+1},$
and taking $\mathcal{F}_{i}$-conditional expectations,
\begin{multline*}
1_{\left\{  \tau^{l_{i}-1}<i<\tau^{l_{i}}\right\}  }\max_{\tau^{l-1}<r\leq
i}(Z_{r}-Y_{r}^{\ast}+\mathcal{S}_{r}-\mathcal{S}_{i})\\
=1_{\left\{  \tau^{l_{i}-1}<i\right\}  }\mathsf{E}_{\mathcal{F}_{i}}\left[
\max_{\tau^{l_{i}-1}<r\leq i}(Z_{r}-Y_{r}^{\ast}+\mathcal{S}_{r}%
-\mathcal{S}_{i+1})1_{\left\{  \tau^{l_{i}-1}<i+1\leq\tau^{l_{i}}\right\}
}\right] \\
\leq1_{\left\{  \tau^{l_{i}-1}<i\right\}  }\mathsf{E}_{\mathcal{F}_{i}}\left[
\max_{\tau^{l_{i+1}-1}<r\leq i+1}(Z_{r}-Y_{r}^{\ast}+\mathcal{S}%
_{r}-\mathcal{S}_{i+1})1_{\left\{  \tau^{l_{i+1}-1}<i+1\leq\tau^{l_{i+1}%
}\right\}  }\right]  \leq0,
\end{multline*}
using the induction hypothesis (\ref{hin}). In view of (\ref{h1}) it follows
that (\ref{c1}) holds for $\tau^{l_{i}-1}<i\leq\tau^{l_{i}}.$

Next, on the other hand, $M\in\mathcal{M}^{\circ}$ implies by
Lemma~\ref{lemopt}-(ii) that for any fixed $l>1$,%
\begin{align}
\max_{\tau^{l-1}\leq r\leq\tau^{l}}(Z_{r}-M_{r}^{\star}+\mathcal{S}_{r})  &
=Z_{\tau^{l-1}}-M_{\tau^{l-1}}^{\star}+\mathcal{S}_{\tau^{l-1}},\text{
\ \ hence}\nonumber\\
\max_{\tau^{l-1}<r\leq\tau^{l}}(Z_{r}-Z_{\tau^{l-1}}+M_{\tau^{l-1}}^{\star
}-M_{r}^{\star}+\mathcal{S}_{r}-\mathcal{S}_{\tau^{l-1}})  &  =0. \label{h4}%
\end{align}
Suppose that $\tau^{l-1}<i\leq\tau^{l}$ and hence $l_{i}=l.$ Then (\ref{h4})
implies by (\ref{ds}) after a few manipulations,%
\begin{gather*}
Z_{i}-Z_{\tau^{l_{i}-1}}+M_{\tau^{l_{i}-1}}^{\star}-M_{i}^{\star}%
+\mathcal{S}_{i}-\mathcal{S}_{\tau^{l_{i}-1}}\\
=\zeta_{\tau^{l_{i}-1}+1}+\mathsf{E}_{\mathcal{F}_{\tau^{l_{i}-1}}}\left[
Y_{\tau^{l_{i}-1}+1}^{\star}\right]  -Z_{\tau^{l_{i}-1}}+Z_{i}-Y_{i}^{\star}\\
+\sum_{r=\tau^{l_{i}-1}+1}^{i-1}\zeta_{r+1}+\sum_{r=\tau^{l_{i}-1}+1}%
^{i-1}\mathsf{E}_{\mathcal{F}_{r}}\left[  Y_{r+1}^{\star}\right]
-\sum_{r=\tau^{l_{i}-1}+1}^{i-1}Y_{r}^{\star}\leq0
\end{gather*}
with the usual convention $\sum_{r=p}^{p-1}:=0.$ Thus, either the last three
sums are zero due to $i=\tau^{l_{i}-1}+1,$ or we may use that $Y_{r}^{\star
}=\mathsf{E}_{\mathcal{F}_{r}}\left[  Y_{r+1}^{\star}\right]  $ for
$\tau^{l_{i}-1}<r<i.$ We thus get for $\tau^{l-1}<i\leq\tau^{l},$%
\begin{equation}
\zeta_{\tau^{l_{i}-1}+1}+\mathsf{E}_{\tau^{l_{i}-1}}\left[  Y_{\tau^{l_{i}%
-1}+1}^{\star}\right]  -Z_{\tau^{l_{i}-1}}+Z_{i}-Y_{i}^{\star}+\mathcal{S}%
_{i}-\mathcal{S}_{\tau^{l_{i}-1}+1}\leq0. \label{h5}%
\end{equation}
In particular, due to $Z_{\tau^{l}}=Y_{\tau^{l}}^{\star},$ for $i=\tau^{l}$
this gives%
\begin{equation}
\zeta_{\tau^{l_{i}-1}+1}+\mathsf{E}_{\mathcal{F}_{\tau^{l_{i}-1}}}\left[
Y_{\tau^{l_{i}-1}+1}^{\star}\right]  -Z_{\tau^{l_{i}-1}}+\mathcal{S}%
_{\tau^{l_{i}}}-\mathcal{S}_{\tau^{l_{i}-1}+1}\leq0. \label{h8}%
\end{equation}
Let us now show that (\ref{c2}) holds for $\tau^{l_{i}-1}<i\leq\tau^{l_{i}}$
and $l_{i}>1$\bigskip\ by backward induction. For $i=\tau^{l_{i}}$ it follows
from (\ref{h8}) by $\zeta_{\tau^{l_{i}-1}+1}-\mathcal{S}_{\tau^{l_{i}-1}%
+1}=-\mathcal{S}_{\tau^{l_{i}-1}}$ that%
\[
Z_{\tau^{l_{i}-1}}-\mathsf{E}_{\mathcal{F}_{\tau^{l_{i}-1}}}\left[
Y_{\tau^{l_{i}-1}+1}^{\star}\right]  +\mathcal{S}_{\tau^{l_{i}-1}}%
-\mathcal{S}_{\tau^{l_{i}}}\geq0
\]
that is (\ref{c2}) for $i=\tau^{l_{i}}.$ Now suppose that for some $i$ with
$\tau^{l_{i}-1}<i<i+1\leq\tau^{l_{i}}$ it holds that%
\[
1_{\left\{  \tau^{l_{i+1}-1}<i+1\leq\tau^{l_{i+1}}\right\}  }\left(
Z_{\tau^{l_{i+1}-1}}-\mathsf{E}_{\mathcal{F}_{\tau^{l_{i+1}-1}}}\left[
Y_{\tau^{l_{i+1}-1}+1}^{\star}\right]  +\mathcal{S}_{\tau^{l_{i+1}-1}%
}-\mathcal{S}_{i+1}\right)  \geq0.
\]
One thus has by construction%

\begin{multline*}
Z_{\tau^{l_{i}-1}}-\mathsf{E}_{\mathcal{F}_{\tau^{l_{i}-1}}}\left[
Y_{\tau^{l_{i}-1}+1}^{\star}\right]  +\mathcal{S}_{\tau^{l_{i}-1}}%
-\mathcal{S}_{i}\\
=Z_{\tau^{l_{i}-1}}-\mathsf{E}_{\mathcal{F}_{\tau^{l_{i}-1}}}\left[
Y_{\tau^{l_{i}-1}+1}^{\star}\right]  +\mathcal{S}_{\tau^{l_{i}-1}}%
-\mathcal{S}_{i+1}+\zeta_{i+1}^{-}.
\end{multline*}
It then follows similarly by taking $\mathcal{F}_{i}$-conditional expectations
that%
\begin{multline*}
1_{\left\{  \tau^{l_{i}-1}<i<\tau^{l_{i}}\right\}  }\left(  Z_{\tau^{l_{i}-1}%
}-\mathsf{E}_{\mathcal{F}_{\tau^{l_{i}-1}}}\left[  Y_{\tau^{l_{i}-1}+1}%
^{\star}\right]  +\mathcal{S}_{\tau^{l_{i}-1}}-\mathcal{S}_{i}\right)
=1_{\left\{  \tau^{l_{i}-1}<i<\tau^{l_{i}}\right\}  }\\
\times\mathsf{E}_{\mathcal{F}_{i}}\left[  \left(  Z_{\tau^{l_{i+1}-1}%
}-\mathsf{E}_{\mathcal{F}_{\tau^{l_{i+1}-1}}}\left[  Y_{\tau^{l_{i+1}-1}%
+1}^{\star}\right]  +\mathcal{S}_{\tau^{l_{i+1}-1}}-\mathcal{S}_{i+1}\right)
1_{\left\{  \tau^{l_{i+1}-1}<i+1\leq\tau^{l_{i+1}}\right\}  }\right]  \geq0
\end{multline*}
by the induction hypothesis (note again that $l_{i+1}=l_{i}$). Thus,
(\ref{c2}) holds for $\tau^{l_{i}-1}<i\leq\tau^{l_{i}}$ and so (\ref{c2}) is
proved. We thus conclude that $\mathcal{S}$ is a martingale that satisfies
(\ref{c1}) and (\ref{c2}). The theorem is proved.

\subsection{Proof of Corollary~\ref{alms}}

Suppose that $M=M^{\star}-\mathcal{S}\in\mathcal{M}^{\circ\circ}$ for some
martingale $\mathcal{S}$ represented by (\ref{ds}). Since $M\in\mathcal{M}%
^{\circ\circ}\subset$ $\mathcal{M}^{\circ},$ Theorem~\ref{cor:main} implies
(via Corollary~\ref{eqco}) that the $\mathcal{\zeta}_{i+1}$ satisfy (\ref{pp})
for $i=\tau^{l_{i}}$. Further, for any $0\leq i\leq J$ one has%
\begin{align*}
Y_{i}^{\star}  &  =\max_{i\leq r\leq J}\left(  Z_{r}-M_{r}+M_{i}\right)
=\max_{i\leq r\leq J}\left(  Z_{r}-M_{r}^{\star}+M_{i}^{\star}+\mathcal{S}%
_{r}-\mathcal{S}_{i}\right) \\
&  \leq Z_{\tau_{i}^{\star}}-M_{\tau_{i}^{\star}}^{\star}+M_{i}^{\star
}+\mathcal{S}_{\tau_{i}^{\star}}-\mathcal{S}_{i}=Y_{i}^{\star}+\mathcal{S}%
_{\tau_{i}^{\star}}-\mathcal{S}_{i}%
\end{align*}
since $M^{\star}\in\mathcal{M}^{\circ\circ}.$ So%
\[
\mathcal{S}_{\tau_{i}^{\star}}-\mathcal{S}_{i}\geq0\text{ \ \ while
\ }\mathsf{E}_{\mathcal{F}_i}\left[  \mathcal{S}_{\tau_{i}^{\star}}%
-\mathcal{S}_{i}\right]  =0,
\]
by Doob's sampling theorem. Hence, by the sandwich property, $\mathcal{S}%
_{\tau_{i}^{\star}}-\mathcal{S}_{i}=0$ for all $0\leq i\leq J.$ This implies
for any $i$ with $\tau^{l-1}<i<\tau^{l}$ that
\[
\mathcal{\zeta}_{i+1}=\mathcal{S}_{i+1}-\mathcal{S}_{i}=\mathcal{S}%
_{\tau_{i+1}^{\star}}-\mathcal{S}_{\tau_{i}^{\star}}=0
\]
due to $\tau_{i}^{\star}=\tau_{i+1}^{\star}=\tau^{l}.$

Conversely, if the $\zeta_{i+1}$ satisfy (\ref{pp}) for $i=\tau^{l_{i}}$ and
further $\zeta_{i+1}=0$ for any $i$ with $\tau^{l_{i-1}}<i<\tau^{l_{i}}%
=\tau_{i}^{\star},$ they also trivially satisfy (\ref{imm}) and (\ref{imp}),
and so one has $M\in\mathcal{M}^{\circ}$ by Theorem~\ref{cor:main} (via
Corollary~\ref{eqco}). Furthermore it follows that $\mathcal{S}_{\tau
_{i}^{\star}}=\mathcal{S}_{i}$ for any $i$ with $\tau^{l_{i-1}}<i<\tau^{l_{i}%
}=\tau_{i}^{\star},$ so by Proposition$~$\ref{prop:propM}
\begin{align*}
\max_{i\leq r\leq J}\left(  Z_{r}-M_{r}\right)   &  =Z_{\tau_{i}^{\star}%
}-M_{\tau_{i}^{\star}}=Z_{\tau_{i}^{\star}}-M_{\tau_{i}^{\star}}^{\star
}+\mathcal{S}_{\tau_{i}^{\star}}\\
&  =Y_{i}^{\star}-M_{i}^{\star}+\mathcal{S}_{\tau_{i}^{\star}}=Y_{i}^{\star
}-M_{i}+\mathcal{S}_{\tau_{i}^{\star}}-\mathcal{S}_{i}\\
&  =Y_{i}^{\star}-M_{i}.
\end{align*}
Hence, $M\in\mathcal{M}^{\circ\circ,i}$ and so $M\in\mathcal{M}^{\circ\circ}$
since $i$ was arbitrary.

\subsection{Proof of Theorem~\ref{i0}}

(i): Due to Proposition$~$\ref{prop:propM}, $M\in\mathcal{M}^{\circ,0}$ if and
only if
\[
0=\max_{0\leq r\leq J}\left(  Z_{r}-M_{r}-Z_{\tau^{\star}}+M_{\tau^{\star}%
}\right)
\]
with $\tau^{\star}:=\tau_{0}^{\star},$ which is equivalent with%
\begin{align}
\max_{0\leq r<\tau^{\star}}\left(  Z_{r}-M_{r}-Z_{\tau^{\star}}+M_{\tau
^{\star}}\right)   &  \leq0\text{ \ \ and}\label{cc1}\\
\max_{\tau^{\star}<r\leq J}\left(  Z_{r}-M_{r}-Z_{\tau^{\star}}+M_{\tau
^{\star}}\right)   &  \leq0. \label{cc2}%
\end{align}
Since $\tau^{\star}=\tau_{r}^{\star}$ for $0\leq r<\tau^{\star},$ (\ref{cc1})
reads
\begin{multline}
\max_{0\leq r<\tau^{\star}}\left(  Z_{r}-M_{r}^{\star}-Z_{\tau_{r}^{\star}%
}+M_{\tau_{r}^{\star}}^{\ast}-\mathcal{S}_{\tau_{r}^{\star}}+\mathcal{S}%
_{r}\right)  =\max_{0\leq r<\tau^{\star}}\left(  Z_{r}-Y_{r}^{\ast
}-\mathcal{S}_{\tau_{r}^{\star}}+\mathcal{S}_{r}\right) \label{ccc1}\\
=\max_{0\leq r<\tau^{\star}}\left(  Z_{r}-Y_{r}^{\ast}-\mathcal{S}%
_{\tau^{\star}}+\mathcal{S}_{r}\right)  \leq0
\end{multline}
which in turn is equivalent with (\ref{dd1}). Indeed, suppose that
(\ref{ccc1}) holds. Then (\ref{dd1}) clearly holds for $j=$ $\tau^{\star}.$
Now assume that (\ref{dd1}) holds for $0<j\leq\tau^{\star}.$ Then, by backward
induction,%
\[
\max_{0\leq r<j-1}\left(  Z_{r}-Y_{r}^{\ast}-\mathcal{S}_{j-1}+\mathcal{S}%
_{r}\right)  =\max_{0\leq r<j-1}\left(  Z_{r}-Y_{r}^{\ast}-\mathcal{S}%
_{j}+\mathcal{S}_{r}\right)  +\zeta_{j}\leq\zeta_{j}%
\]
By next taking $\mathcal{F}_{j-1}$-conditional expectations we get (\ref{dd1})
for $j-1.$ For the converse, just take $j=$ $\tau^{\star}$ in (\ref{dd1}). We
next consider (\ref{cc2}), which may be written as%
\[
\max_{\tau^{\star}<r\leq J}\left(  Z_{r}-M_{r}^{\ast}+M_{\tau^{\star}}^{\ast
}-Z_{\tau^{\star}}-\mathcal{S}_{\tau^{\star}}+\mathcal{S}_{r}\right)  \leq0
\]
Using the Doob decomposition of the Snell envelope (\ref{Doob0}),
$A_{\tau^{\star}}^{\ast}=0,$ and that $Y_{\tau^{\star}}^{\star}=Z_{\tau
^{\star}},$ this is equivalent with (\ref{dd2}).

(ii): Suppose that $M\in\mathcal{M}^{\circ\circ,0}.$ One has that $M=M^{\star
}-\mathcal{S}\in\mathcal{M}^{\circ\circ,0},$ if and only if%
\[
0=\max_{0\leq r\leq J}\left(  Z_{r}-M_{r}-Y_{0}^{\ast}\right)  =\max_{0\leq
r\leq J}\left(  Z_{r}-M_{r}^{\ast}+\mathcal{S}_{r}-Y_{0}^{\ast}\right)  .
\]
Since $Z_{\tau^{\star}}-M_{\tau^{\star}}^{\ast}=Y_{0}^{\ast}$ a.s., this
implies $\mathcal{S}_{\tau^{\star}}\leq0$ a.s., and so by $\mathsf{E}%
_{\mathcal{F}_{0}}\left[  \mathcal{S}_{\tau^{\star}}\right]  =0,$ that
$\mathcal{S}_{\tau^{\star}}=0$ by the sandwich property. Now note that
$\widetilde{\mathcal{S}}_{j}=\mathcal{S}_{j\wedge\tau^{\star}},$ $j=0,...,J,$
is also a martingale with $\widetilde{\mathcal{S}}_{J}=0$ a.s. Let us write
(assuming that $J\geq1$)
\[
0=\widetilde{\mathcal{S}}_{J}=\sum_{j=1}^{J}\widetilde{\mathcal{S}}%
_{j}-\widetilde{\mathcal{S}}_{j-1}=\widetilde{\mathcal{S}}_{J}-\widetilde
{\mathcal{S}}_{J-1}+\sum_{j=1}^{J-1}\widetilde{\mathcal{S}}_{j}-\widetilde
{\mathcal{S}}_{j-1}.
\]
That is, $\widetilde{\mathcal{S}}_{J}-\widetilde{\mathcal{S}}_{J-1}$ is
$\mathcal{F}_{J-1}$-measurable with $\mathsf{E}_{\mathcal{F}_{J-1}}\left[
\widetilde{\mathcal{S}}_{J}-\widetilde{\mathcal{S}}_{J-1}\right]  =0,$ so
$\widetilde{\mathcal{S}}_{J}-\widetilde{\mathcal{S}}_{J-1}=0$ and thus
$\widetilde{\mathcal{S}}_{J-1}=0$ a.s. By proceeding backwards in the same way
we see that $\widetilde{\mathcal{S}}_{j}-\widetilde{\mathcal{S}}_{j-1}=0$ for
all $1\leq j\leq J,$ which implies%
\[
\widetilde{\mathcal{S}}_{j}-\widetilde{\mathcal{S}}_{j-1}=\sum_{r=1}%
^{j\wedge\tau^{\star}}\zeta_{r}-\sum_{r=1}^{(j-1)\wedge\tau^{\star}}\zeta
_{r}=1_{\left\{  \tau^{\star}\geq j\right\}  }\zeta_{j}=0,
\]
\ \ whence $\mathcal{S}_{j}=0$ for $0\leq j\leq\tau^{\star},$ i.e.
(\ref{df1}). Since $\mathcal{M}^{\circ\circ,0}\subset\mathcal{M}^{\circ,0}$
(\ref{df2}) follows from (\ref{dd2}) with $\mathcal{S}_{\tau^{\star}}=0.$
Conversely, if (\ref{df1}) and (\ref{df2}) hold, then%
\begin{align*}
\max_{0\leq r\leq J}\left(  Z_{r}-M_{r}^{\ast}+\mathcal{S}_{r}-Y_{0}^{\ast
}\right)   &  =\max_{0\leq r\leq\tau^{\star}}\left(  Z_{r}-M_{r}^{\ast}%
-Y_{0}^{\ast}\right)  \vee\max_{\tau^{\star}<r\leq J}\left(  Z_{r}-M_{r}%
^{\ast}+\mathcal{S}_{r}-Y_{0}^{\ast}\right) \\
&  =0\vee\max_{\tau^{\star}<r\leq J}\left(  Z_{r}-M_{r}^{\ast}+\mathcal{S}%
_{r}-Y_{0}^{\ast}\right)
\end{align*}
and due to (\ref{df2}), for each $\tau^{\star}<r\leq J$%
\[
Z_{r}-M_{r}^{\ast}+\mathcal{S}_{r}-Y_{0}^{\ast}\leq Y_{r}^{\star}-M_{r}^{\ast
}+A_{r}^{\ast}-Y_{0}^{\ast}=0
\]
by (\ref{Doob0}). That is $\max_{0\leq r\leq J}(Z_{r}-M_{r})=Y_{0}^{\ast}$ and
so $M\in\mathcal{M}^{\circ\circ,0}.$

\subsection{Proof of Proposition~\ref{ASP}}

It holds that%
\begin{align*}
\widetilde{\mathsf{E}}\left[  \max_{0\leq j\leq J}\left(  Z_{j}-\widetilde
{M}_{j}\right)  \right]   &  =\widetilde{\mathsf{E}}\widetilde{\mathsf{E}%
}_{\mathcal{F}}\left[  \max_{0\leq j\leq J}\left(  Z_{j}-M_{j}^{\star
}+\mathcal{S}_{j}+\eta_{j}\right)  \right] \\
&  \geq\widetilde{\mathsf{E}}\left[  \max_{0\leq j\leq J}\left(  Z_{j}%
-M_{j}^{\star}+\mathcal{S}_{j}+\widetilde{\mathsf{E}}_{\mathcal{F}}\left[
\eta_{j}\right]  \right)  \right] \\
&  =\mathsf{E}\left[  \max_{0\leq j\leq J}\left(  Z_{j}-M_{j}^{\star
}+\mathcal{S}_{j}\right)  \right]  \geq Y_{0}^{\star},
\end{align*}
by duality, hence (\ref{uppb}). Further, if $\mathcal{S}=0$ and (\ref{asl})
applies, we may write%
\begin{align}
Z_{j}-\widetilde{M}_{j}  &  =Z_{j}-M_{j}^{\star}+\eta_{j}\nonumber\\
&  =Z_{j}-\left(  Y_{j}^{\star}+A_{j}^{\star}-Y_{0}^{\star}\right)  +\eta
_{j}\nonumber\\
&  =Y_{0}^{\star}+Z_{j}-Y_{j}^{\star}-A_{j}^{\star}+\eta_{j}\leq Y_{0}^{\star}
\label{opa}%
\end{align}
Then (\ref{as}) follows by (\ref{uppb}) and the sandwich property.

As for the last statement: If $\tau^{\star}=0$ one has $Z_{0}=Y_{0}^{\star}$
and $A_{0}^{\star}=0$ by definition, hence in (\ref{asl}) $\eta^{0}\leq0$
a.s., which implies $\eta_{0}=0.$ If $\tau^{\star}>0$ one has $Z_{\tau^{\star
}}=Y_{\tau^{\star}}^{\star}$ and $A_{j}^{\star}-A_{j-1}^{\star}=Y_{j-1}%
^{\star}-\mathsf{E}_{\mathcal{F}_{j-1}}\left[  Y_{j}^{\star}\right]  =0$ for
$j=1,...,\tau^{\star},$ hence $A_{\tau^{\star}}^{\star}=0$ and so $\eta
_{\tau^{\star}}\leq0$ due to (\ref{asl}), implying $\eta_{\tau^{\star}}=0.$ If
$\tau^{\star}$ is strictly optimal, that is $A_{\tau^{\star}+1}^{\star
}=A_{\tau^{\star}+1}^{\star}-A_{\tau^{\star}}^{\star}=Y_{\tau^{\star}}^{\star
}-\mathsf{E}_{\mathcal{F}_{\tau^{\star}}}\left[  Y_{\tau^{\star}+1}^{\star
}\right]  >0,$ one has%
\[
Y_{j}^{\star}-Z_{j}+A_{j}^{\star}>0\text{ \ \ for all }j\neq\tau^{\star}%
\]
since always $Y_{j}^{\star}\geq Z_{j}$ and $A_{j}^{\star}\geq0,$ $Y_{j}%
^{\star}>Z_{j}$\ for $0\leq j<$ $\tau^{\star},$ and $A_{j}^{\star}\geq
A_{\tau^{\star}+1}^{\star}$ for $j>$ $\tau^{\star}$ (remember that $A$ is nondecreasing).

\subsection{Proof of Theorem~\ref{opran}}

Let $M$ $=$ $M^{\star}-\mathcal{S}$ $\in$ $\mathcal{M}^{\circ,0},$ let $\left(
\eta_{j}\right)  $ be as stated, and let us 
assume that
\begin{equation}
\widetilde{\mathsf{E}}\Bigl[\max_{0\leq j\leq J}(Z_{j}-\widetilde{M}%
_{j})\Bigr]=Y_{0}^{\star}. \label{asw1}%
\end{equation}
We then have to show that $M=M^{\star}.$
By
using (\ref{Doob0}) we may write%
\begin{align}
\max_{0\leq j\leq J}\left(  Z_{j}-\widetilde{M}_{j}\right)   &  =\max_{0\leq
j\leq J}\left(  Z_{j}-M_{j}^{\star}+\mathcal{S}_{j}+\eta_{j}\right)
\nonumber\\
&  =Y_{0}^{\star}+\max_{0\leq j\leq J}\left(  \mathcal{S}_{j}+\eta_{j}%
+Z_{j}-Y_{j}^{\star}-A_{j}^{\star}\right)  .\nonumber
\end{align}
By (\ref{asw1}) we must have%
\begin{equation}
\widetilde{\mathsf{E}}\left[  \max_{0\leq j\leq J}\left(  \mathcal{S}_{j}%
+\eta_{j}+Z_{j}-Y_{j}^{\star}-A_{j}^{\star}\right)  \right]  =0. \label{eq}%
\end{equation}
We observe that%
\[
\max_{0\leq j\leq J}\left(  \mathcal{S}_{j}+\eta_{j}+Z_{j}-Y_{j}^{\star}%
-A_{j}^{\star}\right)  \geq\mathcal{S}_{\tau^{\star}}+\eta_{\tau^{\star}%
}+Z_{\tau^{\star}}-Y_{\tau^{\star}}^{\star}-A_{\tau^{\star}}^{\star
}=\mathcal{S}_{\tau^{\star}},
\]
using $\eta_{\tau^{\star}}=0$ due to Proposition~\ref{ASP}. By Doob's sampling
theorem, $\widetilde{\mathsf{E}}\left[  \mathcal{S}_{\tau^{\star}}\right]  =0$
and so (\ref{eq}) implies by the sandwich property,%
\begin{gather}
\max_{0\leq j\leq J}\left(  \mathcal{S}_{j}-\mathcal{S}_{\tau^{\star}}%
+\eta_{j}+Z_{j}-Y_{j}^{\star}-A_{j}^{\star}\right)  =0,\text{ \ \ a.s.,
whence}\nonumber\\
\eta_{j}\leq\mathcal{S}_{\tau^{\star}}-\mathcal{S}_{j}+Y_{j}^{\star}%
-Z_{j}+A_{j}^{\star}\text{ \ \ a.s. for all }0\leq j\leq J. \label{tu1}%
\end{gather}
Let us fix some $0\leq j\leq J$ and assume that $\mathsf{P}(0\leq
j<\tau^{\star})>0.$ Due to (\ref{etu}) we thus have that,
\begin{equation}
\xi_{j}1_{\{0\leq j<\tau^{\star}\}}\leq\left(  1+\frac{\mathcal{S}%
_{\tau^{\star}}-\mathcal{S}_{j}}{Y_{j}^{\star}-Z_{j}+A_{j}^{\star}}\right)
1_{\{0\leq j<\tau^{\star}\}}\text{ \ \ almost surely.} \label{xias}%
\end{equation}
(note that $A_{j}^{\star}\geq0$ and $Y_{j}^{\star}>Z_{j}$ for $0\leq
j<\tau^{\star}$). Since $M\in\mathcal{M}^{\circ,0},$ $0\leq j<\tau^{\star}$
implies by (\ref{dd1}) $\mathcal{S}_{\tau^{\star}}-\mathcal{S}_{j}\geq$
$Z_{j}-Y_{j}^{\star}.$ Now assume that for some $\epsilon>0$ but small enough,
the set
\[
\mathcal{C}_{j}^{\epsilon}:=\left\{  0\leq j<\tau^{\star}\right\}
\cap\left\{  0>-\epsilon(Y_{j}^{\star}-Z_{j}+A_{j}^{\star})>\mathcal{S}%
_{\tau^{\star}}-\mathcal{S}_{j}\geq Z_{j}-Y_{j}^{\star}\right\}
\]
has positive probability. Since on $\mathcal{C}_{j}^{\epsilon}$ one has
\[
1-\frac{Y_{j}^{\star}-Z_{j}}{Y_{j}^{\star}-Z_{j}+A_{j}^{\star}}\leq
1+\frac{\mathcal{S}_{\tau^{\star}}-\mathcal{S}_{j}}{Y_{j}^{\star}-Z_{j}%
+A_{j}^{\star}}<1-\epsilon
\]
we then obtain a contradiction with (\ref{xias}), because $\widetilde
{\mathsf{P}}(\xi_{j}>1-\epsilon)>0.$ Thus for any $\epsilon>0,$ we must have
that $\mathsf{P}(\mathcal{C}_{j}^{\epsilon})=0.$ This in turn implies that%
\[
1_{\{0\leq j<\tau^{\star}\}}\left(  \mathcal{S}_{\tau^{\star}}-\mathcal{S}%
_{j}\right)  \geq0\text{ \ \ a.s.}%
\]
However, the $\mathcal{F}_{j}$-conditional expectation of the left-hand-side
is zero (Doob's sampling theorem). Hence,%
\[
1_{\{0\leq j<\tau^{\star}\}}\mathcal{S}_{\tau^{\star}}=1_{\{0\leq
j<\tau^{\star}\}}\mathcal{S}_{j}\text{ \ \ a.s.}%
\]
by the sandwich property. Since $j$ was arbitrary, this obviously implies
that
\begin{equation}
\mathcal{S}_{j}=0\text{ \ \ for \ \ }\ 0\leq j\leq\tau^{\star}. \label{pa1}%
\end{equation}

Let us next assume that for some $0\leq j\leq J,$ $\mathsf{P}(\tau^{\star
}<j\leq J)>0.$ We then have due to (\ref{etu}) and (\ref{tu1}),%
\begin{equation}
\xi_{j}1_{\{\tau^{\star}<j\leq J\}}\leq\left(  1+\frac{\mathcal{S}%
_{\tau^{\star}}-\mathcal{S}_{j}}{Y_{j}^{\star}-Z_{j}+A_{j}^{\star}}\right)
1_{\{\tau^{\star}<j\leq J\}}\text{ \ \ almost surely.} \label{xias1}%
\end{equation}
For $\tau^{\star}<j\leq J,$ (\ref{dd2}) implies that $\mathcal{S}_{\tau
^{\star}}-\mathcal{S}_{j}\geq$ $Z_{j}-Y_{j}^{\star}-A_{j}^{\star},$ where it
is noted that $Z_{j}-Y_{j}^{\star}-A_{j}^{\star}<0$ due to $Z_{j}$ $\leq$
$Y_{j}$ and $A_{\tau^{\star}+1}^{\star}$ $>$ $0.$ Similarly, we next assume
that for some $\epsilon>0$ the set%
\[
\mathcal{D}_{j}^{\epsilon}:=\left\{  \tau^{\star}<j\leq J\right\}
\cap\left\{  0>-\epsilon(Y_{j}^{\star}-Z_{j}+A_{j}^{\star})>\mathcal{S}%
_{\tau^{\star}}-\mathcal{S}_{j}\geq Z_{j}-Y_{j}^{\star}-A_{j}^{\star}\right\}
\]
has positive probability. Then on $\mathcal{D}_{j}^{\epsilon}$ on has%
\[
0\leq1+\frac{\mathcal{S}_{\tau^{\star}}-\mathcal{S}_{j}}{Y_{j}^{\star}%
-Z_{j}+A_{j}^{\star}}<1-\epsilon,
\]
which gives a contradiction with (\ref{xias1}) however because $\widetilde
{\mathsf{P}}(\xi_{j}>1-\epsilon)>0.$ We so conclude that%
\[
1_{\{\tau^{\star}<j\leq J\}}\left(  \mathcal{S}_{\tau^{\star}}-\mathcal{S}%
_{j}\right)  \geq0\text{ \ \ a.s.}%
\]
and by taking the $\mathcal{F}_{j}$-conditional expectation again, that
$\mathcal{S}_{\tau^{\star}}=\mathcal{S}_{j}$ for $\tau^{\star}\leq j\leq J.$
We had already (\ref{pa1}), and therefore we finally conclude that
$\mathcal{S}=0,$ hence $M=M^{\star}.$

\subsection{Proof of Corollary~\ref{pseudovar}}

If $M=M^{\star}$ one has $\mathrm{Var}\left(  \max_{0\leq j\leq J}%
(Z_{j}-\widetilde{M}_{j})\right)  =0$ due to Proposition~\ref{ASP}. Let us now
take $M\in\mathcal{M}^{\circ,0}$ with $M\neq M^{\star}$ and assume that
$\mathrm{Var}\left(  \max_{0\leq j\leq J}(Z_{j}-\widetilde{M}_{j})\right)
=0.$ From here we will derive a contradiction. As in the proof of
Theorem~\ref{opran} we write%
\begin{align}
\max_{0\leq j\leq J}(Z_{j}-\widetilde{M}_{j})  &  =Y_{0}^{\star}+\max_{0\leq
j\leq J}(\mathcal{S}_{j}+\eta_{j}+Z_{j}-Y_{j}^{\star}-A_{j}^{\star}),\text{
\ \ whence}\nonumber\\
\mathrm{Var}\left(  \max_{0\leq j\leq J}(Z_{j}-\widetilde{M}_{j})\right)   &
=\mathrm{Var}\left(  \max_{0\leq j\leq J}(\mathcal{S}_{j}+\eta_{j}+Z_{j}%
-Y_{j}^{\star}-A_{j}^{\star})\right)  =0. \label{va}%
\end{align}
Now, $M\neq M^{\star}$ implies by Theorem~\ref{opran} that%
\begin{equation}
\widetilde{\mathsf{E}}\left[  \max_{0\leq j\leq J}(\mathcal{S}_{j}+\eta
_{j}+Z_{j}-Y_{j}^{\star}-A_{j}^{\star})\right]  >0. \label{ea}%
\end{equation}
That is, due to (\ref{va}) and (\ref{ea}), there exists a constant $c>0$ such
that%
\[
\max_{0\leq j\leq J}(\mathcal{S}_{j}+\eta_{j}+Z_{j}-Y_{j}^{\star}-A_{j}%
^{\star})=c>0.
\]
Using (\ref{etu}) and the fact that always $Y_{j}^{\star}-Z_{j}+A_{j}^{\star
}\geq0$ and $\xi_{j}\leq1,$ this implies%
\begin{equation}
0<c=\max_{0\leq j\leq J}(\mathcal{S}_{j}+\left(  \xi_{j}-1\right)
(Y_{j}^{\star}-Z_{j}+A_{j}^{\star}))\leq\max_{0\leq j\leq J}(\mathcal{S}_{j}).
\label{ceq}%
\end{equation}
Consider the stopping time $\sigma:=\inf\{j\geq0:\mathcal{S}_{j}\geq c\}.$
Then, using $\mathcal{S}_{0}=0$ and (\ref{ceq}), we must have that
$0<\sigma\leq J$ almost surely. Since $\mathcal{S}$ is a martingale, Doob's
sampling theorem then implies $0=\mathcal{S}_{0}=\mathsf{E}\left[
\mathcal{S}_{\sigma}\right]  \geq c,$ hence a contradiction. That is, the
assumption $\mathrm{Var}\left(  \max_{0\leq j\leq J}(Z_{j}-\widetilde{M}%
_{j})\right)  =0$ was false.

\section{Numerical examples}

\label{SecNum}

\subsection{Simple stylized numerical example}

\label{styex}

We first reconsider the stylized test example due to \cite[Section
8]{J_SchZhaHua}, also considered in \cite{J_BelHilSch}, where $J=2$, $Z_{0}%
=0$, $Z_{2}=1$, and $Z_{1}=\mathcal{U}$ is a random variable which uniformly
distributed on the interval $[0,2]$. The optimal stopping time $\tau^{\ast}$
is thus given by
\[
\tau^{\ast}=\left\{
\begin{array}
[c]{rcl}%
1, & \quad & \mathcal{U}\geq1,\\
2, & \quad & \mathcal{U}<1.
\end{array}
\right.
\]
and the optimal value is $Y_{0}^{\star}=\mathsf{E}\max(\mathcal{U},1)=5/4$.
Furthermore, it is easy to see that the Doob martingale is given by
\[
M_{0}^{\star}=0,\quad M_{1}^{\star}=M_{2}^{\star}=\max\{\mathcal{U}%
,1\}-\frac{5}{4}.
\]
As an illustration of the theory developed in Sections~\ref{SecCh}%
-\ref{SecRan}, let us consider the linear span $M\left(  \alpha\right)
=\alpha M^{\star}$ as a pool of candidate martingales and randomize it
according to (\ref{etu}). We thus consider the objective function
\begin{equation}
\mathcal{O}_{\theta}(\alpha):=\widetilde{\mathsf{E}}\Bigl[\max_{0\leq j\leq
2}\left(  Z_{j}-\alpha M_{j}^{\star}+\theta\xi_{j}\left(  Y_{j}^{\star}%
-Z_{j}+A_{j}^{\star}\right)  \right)  \Bigr], \label{expob}%
\end{equation}
for some fixed $\theta\geq0,$ where $(\xi_{j})$ are i.i.d. random variables
with uniform distribution on $[-1,1].$ Note that for this example
$Y_{1}^{\star}=\max(\mathcal{U},1),$ $Y_{2}^{\star}=1,$ and $A_{0}^{\star
}=A_{1}^{\star}=0,$ $A_{2}^{\star}=\max\{\mathcal{U},1\}-1,$ is the
non-decreasing predictable process from the Doob decomposition. Moreover, it
is possible to compute (\ref{expob}) in closed form (though we omit detailed
expressions which can be conveniently obtained by Mathematica for instance).
In Figure~\ref{thet0} (left panel) we have plotted (\ref{expob}) for
$\theta=0$ and $\theta=1,$ together with the objective function%
\[
\overline{\mathcal{O}}_{1}(\alpha):=\widetilde{\mathsf{E}}\Bigl[\max_{0\leq
j\leq2}\left(  Z_{j}-\alpha M_{j}^{\star}+\xi_{j}\right)  \Bigr],
\]
due to a \textquotedblleft naive\textquotedblright\ randomization, not based
on knowledge of the factor $Y_{j}^{\star}-Z_{j}+A_{j}^{\star}.$ Also, in
Figure~\ref{thet0} (right panel), the relative standard deviations
$\sqrt{\mathrm{Var}(\cdot)}/Y_{0}^{\star}$ of the corresponding random
variables%
\begin{align*}
\mathcal{Z}_{\theta}(\alpha)  &  :=\max_{0\leq j\leq2}\left(  Z_{j}-\alpha
M_{j}^{\star}+\theta\xi_{j}\left(  Y_{j}^{\star}-Z_{j}+A_{j}^{\star}\right)
\right)  ,\text{ \ \ }\theta=0,1,\text{ \ and}\\
\overline{\mathcal{Z}}_{1}(\alpha)  &  :=\max_{0\leq j\leq2}\left(
Z_{j}-\alpha M_{j}^{\star}+\xi_{j}\right)
\end{align*}
are depicted as a function of $\alpha.$

\begin{center}
\begin{figure}
\includegraphics[width= 6cm]{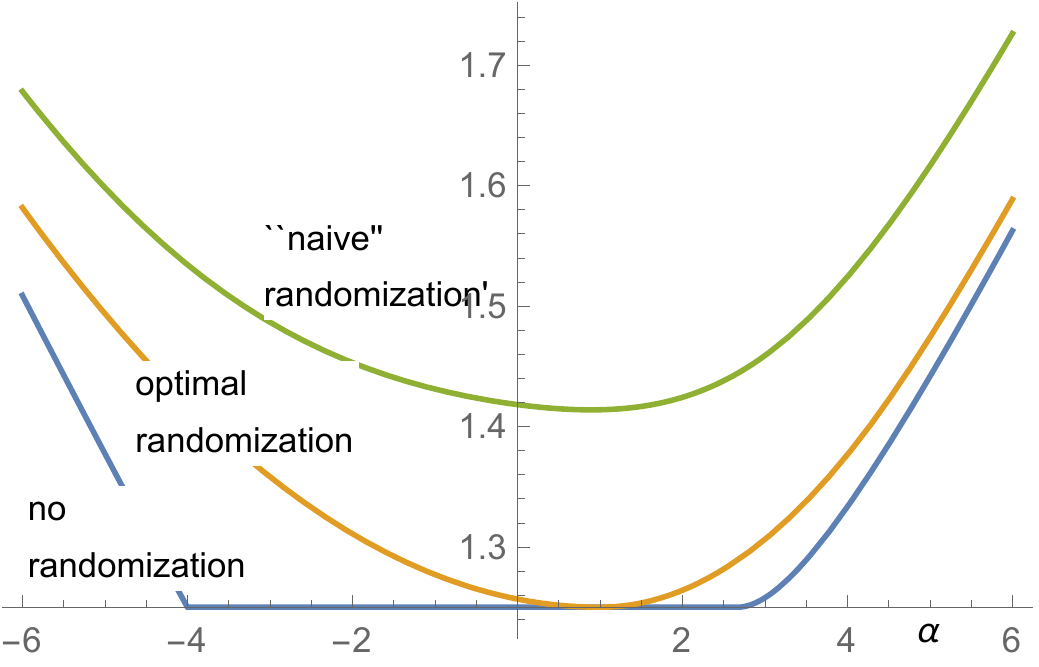}\quad
\includegraphics[width= 6cm]{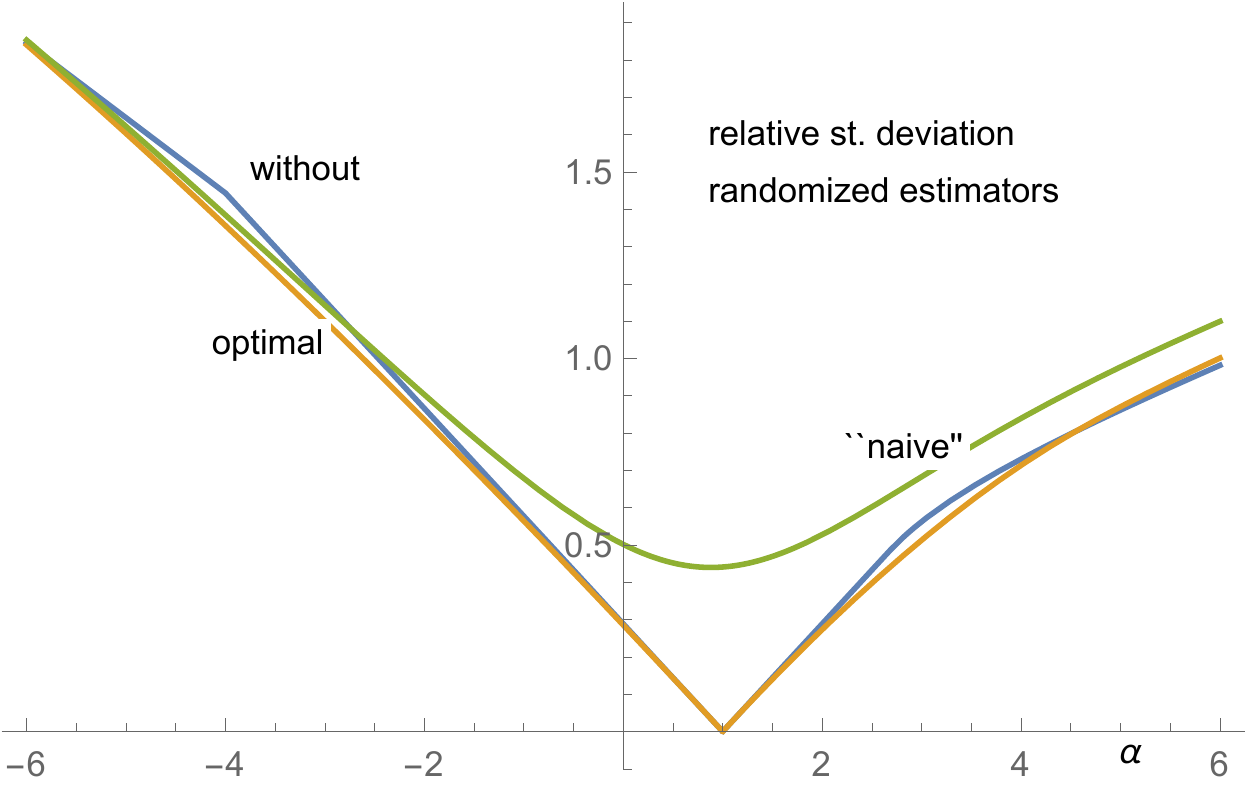}\caption{Left panel:
objective functions $\mathcal{O}_{0}(\alpha)$ (no randomization),
$\mathcal{O}_{1}(\alpha)$ (optimal randomization), and $\overline{\mathcal{O}%
}_{1}$ (``naive'' randomization); right panel: relative deviations of
$\mathcal{Z}_{0}(\alpha)$ (without randomization), $\mathcal{Z}_{1}(\alpha)$
(optimal randomization), $\overline{\mathcal{Z}}_{1}(\alpha)$ (``naive''
randomization)}%
\label{thet0}%
\end{figure}
\end{center}
From \cite[Section 8]{J_SchZhaHua} we know that, and from the plot of
$\mathcal{O}_{0}(\alpha)$ in Figure~\ref{thet0} (left panel) we see that,
$M(\alpha)\in\mathcal{M}_{0}^{\circ}$ for $\alpha\in\lbrack-4,8/3]$. On the
other hand, the right panel plot shows that $\mathrm{Var}(\mathcal{Z}%
_{0}(\alpha))$ may be relatively large for $\alpha\neq1,$ and that the Doob
martingale (i.e. $\alpha=1$) is the only surely optimal one in our parametric
family. Moreover, the objective function due to the optimal randomization
attains its unique minimum at the Doob martingale, i.e. for $\alpha=1.$
Further, the variance of the corresponding optimally randomized estimator
attains its unique minimum zero also at $\alpha=1.$ Let us note that these
observations are anticipated by Theorem~\ref{opran} and
Corollary~\ref{pseudovar}. The catch is that for each $\alpha\neq1$ the
randomized $M(\alpha)$ fails to be optimal in the sense of (\ref{asw}). We
also see that both the optimal and the \textquotedblleft
naive\textquotedblright\ randomization render the minimization problem to be
strictly convex. Moreover, while the minimum due to the \textquotedblleft
naive\textquotedblright\ randomization lays significantly above the true
solution, the argument where the minimum is attained, $\overline{\alpha}$ say,
identifies nonetheless a martingale that virtually coincides with the Doob
optimal one. That is, $\overline{\alpha}\approx1$ and $M(\overline{\alpha})$
is optimal corresponding to variance $\mathrm{Var}(\mathcal{Z}_{0}%
(\overline{\alpha}))\approx0$, which can be seen in the right panel.

\subsection{Bermudan call in a Black-Scholes model}

In order to exhibit the merits of randomization based on the theoretical
results in this paper in a more realistic case, we have constructed an example
that contains all typical features of a real life Bermudan option, but, is
simple enough to be treated numerically in all respects on the other hand.

As in the previous example we take $J=2$, and specify the (discounted)
cash-flows $Z_{j}$ as functions of the (discounted) stock prices $S_{j}$ by%
\begin{equation}
Z_{0}=0,\text{ \ \ }Z_{1}=(S_{1}-\kappa_{1})^{+},\text{ \ \ }Z_{2}%
=(S_{2}-\kappa_{2})^{+} \label{str}%
\end{equation}
For $S$ we take the Black-Scholes model%
\begin{equation}
S_{j}=S_{0}\exp(-\frac{1}{2}\sigma^{2}j+\sigma W_{j}),\text{ \ \ }j=0,1,2,
\label{St}%
\end{equation}
where $W_{1}\sim\mathcal{N}\left(  0,1\right)  $ and $W_{1,2}:=W_{2}-W_{1}%
\sim\mathcal{N}\left(  0,1\right)  ,$ independent of $W_{1}.$ As such we have
a stylized example of a Bermudan call option under a Black-Scholes model with
two (non-trivial) exercise dates if $\kappa_{2}>\kappa_{1}\geq0$. Note that
usually a Bermudan call is considered for a fixed strike and a dividend paying
stock, yielding a non-trivial optimal stopping time. Though increasing strikes
here look somewhat unusual, it is simple for presentation while,
mathematically, the effect is the same as for a dividend paying stock and a
fixed strike. For the continuation function at $j=1$ we thus have%
\begin{align}
C_{1}(W_{1})  &  =\mathsf{E}_{W_{1}}\left[  \left(  S_{0}\exp(-\sigma
^{2}+\sigma W_{2})-\kappa_{2}\right)  ^{+}\right] \nonumber\\
&  =\int\left(  S_{0}\exp(-\sigma^{2}+\sigma W_{1}+\sigma z)-\kappa
_{2}\right)  ^{+}\phi(z)dz, \label{bs}%
\end{align}
where $\phi(z)=\left(  2\pi\right)  ^{-1/2}\exp(-z^{2}/2)$ is the standard
normal density. While abusing notation a bit we will denote the cash-flows by
$Z_{1}(W_{1})$ and $Z_{2}(W_{2})=Z_{2}(W_{1,}W_{1,2}),$ respectively. For the
(discounted) option value at $j=0$ one thus has%
\begin{align*}
Y_{0}^{\star}  &  =\mathsf{E}\left[  \max\left(  Z_{1}(W_{1}),C_{1}%
(W_{1})\right)  \right] \\
&  =\int\max\left(  \left(  S_{0}\exp(-\frac{1}{2}\sigma^{2}+\sigma
z)-\kappa_{1}\right)  ^{+},C_{1}(z)\right)  \phi(z)dz
\end{align*}
Further we obviously have%
\[
Y_{1}^{\star}(W_{1})=\max\left(  Z_{1}(W_{1}),C_{1}(W_{1})\right)  \text{
\ \ and \ \ }Y_{2}^{\star}(W_{2})=Z_{2}(W_{2})=Z_{2}(W_{1,}W_{1,2}).
\]
The Doob martingale for this example is thus given by%
\[
M_{0}^{\star}=0,\text{ \ \ }M_{1}^{\star}=Y_{1}^{\star}(W_{1})-Y_{0}^{\star
},\text{ \ \ }M_{2}^{\star}-M_{1}^{\star}=Z_{2}(W_{1,}W_{1,2})-C_{1}(W_{1})
\]
and the non-decreasing predictable component $A^{\star}$ is given by%
\[
A_{0}^{\star}=A_{1}^{\star}=0,\text{ \ \ }A_{2}^{\star}=Y_{1}^{\star}%
(W_{1})-C_{1}(W_{1}).
\]
For demonstration purposes we will quasi analytically compute the optimal
randomization coefficient in (\ref{etu}),%
\[
Y^{\star}-Z+A^{\star}=
\begin{cases}
Y_{0}^{\star} & j=0,\\
(C_{1}(W_{1})-Z_{1}(W_{1}))^{+}, & j=1,\\
( Z_{1}(W_{1})-C_{1}(W_{1}))^{+}, & j=2.
\end{cases}
\]
by using a Black(-Scholes) type formula%
\begin{multline*}
C_{1}(W_{1})=S_{0}\exp(-\frac{1}{2}\sigma^{2}+\sigma W_{1})\mathcal{N}\left(
W_{1}+\frac{1}{\sigma}\ln(S_{0}/\kappa_{2})\right)  \\
-\kappa_{2}\mathcal{N}\left(  W_{1}+\frac{1}{\sigma}\ln(S_{0}/\kappa
_{2})-\sigma\right)  ,
\end{multline*}
and a numerical integration for obtaining the target value $Y_{0}^{\star}$.
We now consider two martingale families.
\begin{description}
\item [(M-Sty)] For any
$\boldsymbol{\alpha}=(\alpha_{11},\alpha_{12},\alpha_{21}%
,\alpha_{22})$ we  set%
\begin{align}
M_{1}^{\text{sty}}(\boldsymbol{\alpha},W)  &  :=\alpha_{11}\left(
Y_{1}^{\star}(W_{1})-Y_{0}^{\star}-W_{1}\right)  +\alpha_{12}W_{1}%
\label{Msty}\\
M_{2}^{\text{sty}}(\boldsymbol{\alpha},W)  &  :=M_{1}^{\text{sty}%
}(\boldsymbol{\alpha},W)+\alpha_{21}\left(  Z_{2}(W_{1,}W_{1,2})-C_{1}%
(W_{1})-W_{1,2}\right)  +\alpha_{22}W_{1,2}.\nonumber
\end{align}
Note that $M^{\text{sty}}((1,1,1,1),W)=M^{\star}(W).$

\item [(M-Hermite)] Using that the (probabilistic) Hermite polynomials given by
\[
He_{k}(x)=(-1)^{k}e^{\frac{x^{2}}{2}}\left(  \frac{d}{dx}\right)
^{k}e^{-\frac{x^{2}}{2}},\text{ \ \ }k=0,1,2,...,
\]
are orthogonal with respect to the standard Gaussian density we consider a martingale family
\begin{align}
M_{1}^{\text{H}}\left(  \boldsymbol{\alpha},W\right)   &  =\sum_{k=1}%
^{K}\alpha_{1,k}He_{k}(W_{1})\label{MH}\\
M_{2}^{\text{H}}\left(  \boldsymbol{\alpha},W\right)   &  =M_{1}^{H}\left(
\boldsymbol{\alpha},W\right)  +\sum_{k=0}^{K}\sum_{l=1}^{L}\alpha
_{2,k,l}He_{k}(W_{1})He_{l}(W_{1,2}),\nonumber
\end{align}
with obvious definition of $\boldsymbol{\alpha}\in\mathbb{R}^{K}%
\oplus\,\mathbb{R}^{(K+1)}\times\mathbb{R}^{L}$ (note that $He_{0}\equiv1$).
Since our mere goal is to exhibit the effect of randomization, for the
examples below we restrict ourselves to the choice $K=L=3.$
\end{description}
The parameters in (\ref{str}) and (\ref{St}) are taken to be such that with a
medial probability optimal exercise takes place at $j=1.$ In particular, we
consider two cases specified with parameter sets%
\begin{align*}
\text{(Pa1)}  &  \text{: \ }S_{0}=2,\text{ \ \ }\sigma^{2}=\frac{1}{3},\text{
\ \ }\kappa_{1}=2,\text{ \ \ }\kappa_{2}=3,\text{ \ \ target value }%
Y_{0}^{\star}=0.164402, \\
\text{(Pa2)}  &  \text{: \ }S_{0}=2,\text{ \ \ }\sigma^{2}=\frac{1}{25},\text{
\ \ }\kappa_{1}=2,\text{ \ \ }\kappa_{2}=\frac{5}{2},\text{ \ \ target value
}Y_{0}^{\star}=0.496182,
\end{align*}
respectively. From Figure~\ref{exb} we see that the probability of optimal
exercise at $j=1$ is almost 50\% for (Pa1) and almost 30\% for (Pa2). Let us
visualize on the basis of martingale family (M-Sty) and parameters (Pa1) the
effects of randomization. Consider the objective function%
\begin{equation}
\mathcal{O}_{\theta}(\boldsymbol{\alpha}):=\widetilde{\mathsf{E}}%
\Bigl[\max_{0\leq j\leq2}\left(  Z_{j}-M_{j}^{\text{sty}}(\boldsymbol{\alpha
})+\theta\xi_{j}\left(  Y_{j}^{\star}-Z_{j}+A_{j}^{\star}\right)  \right)
\Bigr]. \label{obex2}%
\end{equation}
where $\theta$ scales the randomization due to i.i.d. random variables
$(\xi_{j}),$ uniformly distributed on $[-1,1]$. I.e., for $\theta=0$ there is
no randomization and $\theta=1$ gives the optimal randomization. Now restrict
(\ref{obex2}) to the sub domain $\boldsymbol{\alpha}=(\alpha_{1},\alpha
_{1},\alpha_{2},\alpha_{2})=:(\alpha_{1},\alpha_{2})$ (while slightly abusing
notation), i.e. $\alpha_{11}=\alpha_{12}=\alpha_{1}$ and $\alpha_{21}%
=\alpha_{22}=\alpha_{2}.$ The function $\mathcal{O}_{0}(\alpha_{1},\alpha
_{2}),$ i.e. (\ref{obex2}) without randomization is visualized in
Figure~\ref{plat}, where expectations are computed quasi-analytically with
Mathematica. From this plot we see that the true value $Y_{0}^{\star
}=0.164402$ is attained on the line $(\alpha_{1},1)$ for various $\alpha_{1}$
(i.e. not only in $(1,1)$). On the other hand, $\mathcal{O}_{1}(\alpha
_{1},\alpha_{2})$ i.e. (\ref{obex2}) with optimal randomization, has a clear
strict global minimum in $(1,1)$, see Figure~\ref{platran}. Let us have a
closer look at the map $\alpha_{1}\rightarrow\mathcal{O}_{\theta}(\alpha
_{1},\alpha_{1},1,1)$ for $\theta=0$ and $\theta=1,$ respectively, and also at
$\alpha_{1}\rightarrow\overline{\mathcal{O}}_{0.16}(\alpha_{1},\alpha
_{1},1,1)$ due to the \textquotedblleft naive\textquotedblright\ randomization%
\[
\overline{\mathcal{O}}_{0.16}(\alpha_{1},1):=\widetilde{\mathsf{E}}%
\Bigl[\max_{0\leq j\leq2}\left(  Z_{j}-M_{j}^{\text{sty}}(\alpha
_{1},1)+0.16\,\xi_{j}\right)  \Bigr],
\]
where the scale parameter $\theta=0.16$ is taken to be roughly the option
value. (It turns out that the choice of this scale factor is not critical for
the location of the minimum.) In fact, the results, plotted in
Figure~\ref{ranvar}, tell there own tale. The second panel depicts the
relative deviation of%
\[
\mathcal{Z}_{0}(\alpha_{1},1):=\max_{0\leq j\leq2}\left(  Z_{j}-M_{j}%
^{\text{sty}}(\alpha_{1},1)\right).
\]
In fact, similar comments as for the example in Section~\ref{styex} apply. The
\textquotedblleft naive\textquotedblright\ randomization attains its minimum
at $\overline{\alpha}_{1}=0.9,$ which we red off from the tables that
generated this figure. We thus have found the martingale $M^{\text{sty}%
}(0.9,1),$ which may be virtually considered surely optimal, as can be seen
from the variance plot (second panel). Analogue visualizations for the
parameter set (Pa2) with analogue conclusions may be given, though are omitted
due to space restrictions.

\medskip

Let us now pass on to a Monte Carlo setting, where we mimic the approach in
real practice more closely. Based on $N$ simulated samples of the underlying
asset model, i.e. $S^{(n)},$ $n=1,...,N,$ we consider the minimization
\begin{equation}
\widehat{\boldsymbol{\alpha}}_{\theta}:=\underset{\boldsymbol{\alpha}}%
{\arg\min}\frac{1}{N}\sum_{n=1}^{N}\Bigl[\max_{0\leq j\leq2}\left(
Z_{j}^{(n)}-M_{j}^{(n)}(\boldsymbol{\alpha})+\theta\xi_{j}\left(  Y_{j}%
^{\star(n)}-Z_{j}^{(n)}+A_{j}^{\star(n)}\right)  \right)  \Bigr] \label{minth}%
\end{equation}
for $\theta=0$ (no randomization) and $\theta=1$ (optimal randomization),
along with the minimization%
\begin{equation}
\widehat{\boldsymbol{\alpha}}_{\theta^{\text{naive}}}:=\underset
{\boldsymbol{\alpha}}{\arg\min}\frac{1}{N}\sum_{n=1}^{N}\Bigl[\max_{0\leq
j\leq2}\left(  Z_{j}^{(n)}-M_{j}^{(n)}(\boldsymbol{\alpha})+\theta
_{j}^{\text{naive}}\xi_{j}\right)  \Bigr] \label{minpr}%
\end{equation}
based on a \textquotedblleft naive\textquotedblright randomization where the
coefficients $\theta_{j}^{\text{naive}},$ $j=0,1,2$ are pragmatically chosen.
In (\ref{minth}) and (\ref{minpr}) $M$ stands for a generic linearly
structured martingale family, such as (\ref{Msty}) and (\ref{MH}) for example.
The minimization problems (\ref{minth}) and (\ref{minpr}) may be solved by
linear programming (LP). They may be transformed into a suitable form such
that the (free) LP package in R can be applied. This transformation procedure
is straightforward and spelled out in \cite{J_DesFarMoa} for example. In the
latter paper it is argued that the required computation time scales with $N$
due to the sparse structure of the coefficient matrix involved in the LP
setup. However, taking advantage of this sparsity requires a special treatment
of the implementation of the linear program in connection with more advanced
LP solvers (as done in \cite{J_DesFarMoa}). Since this paper is essentially on
the theoretical justification of the randomized duality problem (along with
the classification of optimal martingales), we consider an in-depth numerical
analysis beyond scope of this paper.

For both parameter sets (Pa1) and (Pa2), and both martingale families
(\ref{Msty}) and (\ref{MH}) with $K=L=3,$ we have carried out the LP
optimization algorithm sketched above. We have taken $N=2000$ and for the
\textquotedblleft naive\textquotedblright\ randomization%
\[
\theta_{0}^{\text{naive}}=1.6\text{ for (Pa1), \ }\theta_{0}^{\text{naive}%
}=4.8\text{ for (Pa2), and simply }\theta_{1}^{\text{naive}}=\theta
_{2}^{\text{naive}}=0.
\]
In the Table~\ref{tablePa1}, for (Pa1), and Table~\ref{tablePa2}, for (Pa2),
we present for the minimizers $\widehat{\boldsymbol{\alpha}}_{0}%
,\widehat{\boldsymbol{\alpha}}_{1},\widehat{\boldsymbol{\alpha}}%
_{\theta^{\text{naive}}}$ the in-sample expectation $\widehat{m}$, the
in-sample standard deviation $\widehat{\sigma}/\sqrt{N},$ and the path-wise
maximum due to a single trajectory $\widehat{\sigma},$ followed by the
corresponding \textquotedblleft true\textquotedblright\ values $m^{\text{test}%
},$ $\sigma^{\text{test}}/\sqrt{N^{\text{test}}},$ $\sigma^{\text{test}},$
based on a large \textquotedblleft test\textquotedblright\ simulation of
$N^{\text{test}}=10^{6}$ samples.

\begin{table}[th]
\centering
$%
\begin{tabular}
[c]{|c|c|c|c|c|c|c|}\hline
(Pa1) & \multicolumn{3}{|c}{$M^{\text{sty}}$} &
\multicolumn{3}{|c|}{$M^{\text{H}}$}\\\hline
\multicolumn{1}{|l|}{} & $\widehat{\boldsymbol{\alpha}}_{0}$ & $\widehat
{\boldsymbol{\alpha}}_{\theta^{\text{naive}}}$ & $\widehat{\boldsymbol{\alpha
}}_{1}$ & $\widehat{\boldsymbol{\alpha}}_{0}$ & $\widehat{\boldsymbol{\alpha}%
}_{\theta^{\text{naive}}}$ & $\widehat{\boldsymbol{\alpha}}_{1}$\\\hline
\multicolumn{1}{|l|}{$\widehat{m}$} & \multicolumn{1}{|l|}{$0.16243$} &
\multicolumn{1}{|l|}{$0.16399$} & \multicolumn{1}{|l|}{$0.16403$} &
\multicolumn{1}{|l|}{$0.16268$} & \multicolumn{1}{|l|}{$0.16560$} &
\multicolumn{1}{|l|}{$0.16696$}\\\hline
\multicolumn{1}{|l|}{$\widehat{\sigma}/\sqrt{N}$} &
\multicolumn{1}{|l|}{$0.00573$} & \multicolumn{1}{|l|}{$0.00036$} &
\multicolumn{1}{|l|}{$0.00029$} & \multicolumn{1}{|l|}{$0.00574$} &
\multicolumn{1}{|l|}{$0.00113$} & \multicolumn{1}{|l|}{$0.00118$}\\\hline
\multicolumn{1}{|l|}{$\widehat{\sigma}$} & \multicolumn{1}{|l|}{$0.25639$} &
\multicolumn{1}{|l|}{$0.01608$} & \multicolumn{1}{|l|}{$0.01278$} &
\multicolumn{1}{|l|}{$0.25676$} & \multicolumn{1}{|l|}{$0.05063$} &
\multicolumn{1}{|l|}{$0.05293$}\\\hline
\multicolumn{1}{|l|}{$m^{\text{test}}$} & \multicolumn{1}{|l|}{$0.16490$} &
\multicolumn{1}{|l|}{$0.16445$} & \multicolumn{1}{|l|}{$0.16442$} &
\multicolumn{1}{|l|}{$0.16709$} & \multicolumn{1}{|l|}{$0.16664$} &
\multicolumn{1}{|l|}{$0.16685$}\\\hline
\multicolumn{1}{|l|}{$\sigma^{\text{test}}/\sqrt{N^{\text{test}}}$} &
\multicolumn{1}{|l|}{$0.00026$} & \multicolumn{1}{|l|}{$0.00001$} &
\multicolumn{1}{|l|}{$0.00001$} & \multicolumn{1}{|l|}{$0.00026$} &
\multicolumn{1}{|l|}{$0.00005$} & \multicolumn{1}{|l|}{$0.00005$}\\\hline
\multicolumn{1}{|l|}{$\sigma^{\text{test}}$} & \multicolumn{1}{|l|}{$0.26096$}
& \multicolumn{1}{|l|}{$0.01460$} & \multicolumn{1}{|l|}{$0.01064$} &
\multicolumn{1}{|l|}{$0.26439$} & \multicolumn{1}{|l|}{$0.05083$} &
\multicolumn{1}{|l|}{$0.05153$}\\\hline
\end{tabular}
$ \caption{LP minimization results due to $M^{\text{sty}}$ and $M^{\text{H}}$
for (Pa1)}%
\label{tablePa1}%
\end{table}

\begin{table}[th]
\centering
$%
\begin{tabular}
[c]{|c|c|c|c|c|c|c|}\hline
(Pa2) & \multicolumn{3}{|c}{$M^{\text{sty}}$} &
\multicolumn{3}{|c|}{$M^{\text{H}}$}\\\hline
\multicolumn{1}{|l|}{} & $\widehat{\boldsymbol{\alpha}}_{0}$ & $\widehat
{\boldsymbol{\alpha}}_{\theta^{\text{naive}}}$ & $\widehat{\boldsymbol{\alpha
}}_{1}$ & $\widehat{\boldsymbol{\alpha}}_{0}$ & $\widehat{\boldsymbol{\alpha}%
}_{\theta^{\text{naive}}}$ & $\widehat{\boldsymbol{\alpha}}_{1}$\\\hline
\multicolumn{1}{|l|}{$\widehat{m}$} & \multicolumn{1}{|l|}{$0.48748$} &
\multicolumn{1}{|l|}{$0.49471$} & \multicolumn{1}{|l|}{$0.49490$} &
\multicolumn{1}{|l|}{$0.49329$} & \multicolumn{1}{|l|}{$0.50082$} &
\multicolumn{1}{|l|}{$0.50546$}\\\hline
\multicolumn{1}{|l|}{$\widehat{\sigma}/\sqrt{N}$} &
\multicolumn{1}{|l|}{$0.02064$} & \multicolumn{1}{|l|}{$0.00201$} &
\multicolumn{1}{|l|}{$0.00152$} & \multicolumn{1}{|l|}{$0.02076$} &
\multicolumn{1}{|l|}{$0.00318$} & \multicolumn{1}{|l|}{$0.00308$}\\\hline
\multicolumn{1}{|l|}{$\widehat{\sigma}$} & \multicolumn{1}{|l|}{$0.92301$} &
\multicolumn{1}{|l|}{$0.08981$} & \multicolumn{1}{|l|}{$0.06801$} &
\multicolumn{1}{|l|}{$0.92852$} & \multicolumn{1}{|l|}{$0.14222$} &
\multicolumn{1}{|l|}{$0.13762$}\\\hline
\multicolumn{1}{|l|}{$m^{\text{test}}$} & \multicolumn{1}{|l|}{$0.49820$} &
\multicolumn{1}{|l|}{$0.49639$} & \multicolumn{1}{|l|}{$0.49633$} &
\multicolumn{1}{|l|}{$0.51079$} & \multicolumn{1}{|l|}{$0.50870$} &
\multicolumn{1}{|l|}{$0.50912$}\\\hline
\multicolumn{1}{|l|}{$\sigma^{\text{test}}/\sqrt{N^{\text{test}}}$} &
\multicolumn{1}{|l|}{$0.00095$} & \multicolumn{1}{|l|}{$0.00009$} &
\multicolumn{1}{|l|}{$0.00007$} & \multicolumn{1}{|l|}{$0.00097$} &
\multicolumn{1}{|l|}{$0.00016$} & \multicolumn{1}{|l|}{$0.00015$}\\\hline
\multicolumn{1}{|l|}{$\sigma^{\text{test}}$} & \multicolumn{1}{|l|}{$0.95415$}
& \multicolumn{1}{|l|}{$0.09038$} & \multicolumn{1}{|l|}{$0.06674$} &
\multicolumn{1}{|l|}{$0.97272$} & \multicolumn{1}{|l|}{$0.16047$} &
\multicolumn{1}{|l|}{$0.15103$}\\\hline
\end{tabular}
$ \caption{LP minimization results due to $M^{\text{sty}}$ and $M^{\text{H}}$
for (Pa2)}%
\label{tablePa2}%
\end{table}

The results in tables Tables~\ref{tablePa1}-\ref{tablePa2} show  that
even  a simple (naive) randomization at $j=0$ leads to a substantial variance reduction
(up to \(10\) times) not only on training samples but also on the test ones.
We think that for more  structured examples and more complex families of martingales
 even more pronounced variance reduction effect may be expected. For example, in general it might be better to take Wiener integrals, i.e.
 objects of the form $\int\alpha(t,X_t)dW,$ where $\alpha$ runs through some linear space of basis functions,
 as building blocks for the  martingale family.
 Also other types of randomization can be used, for example one may take different distributions for the r.v. \(\xi.\) However all these issues
 will be analyzed in  a subsequent study.

\begin{center}
\begin{figure}[tbh]
\includegraphics[width= 6cm]{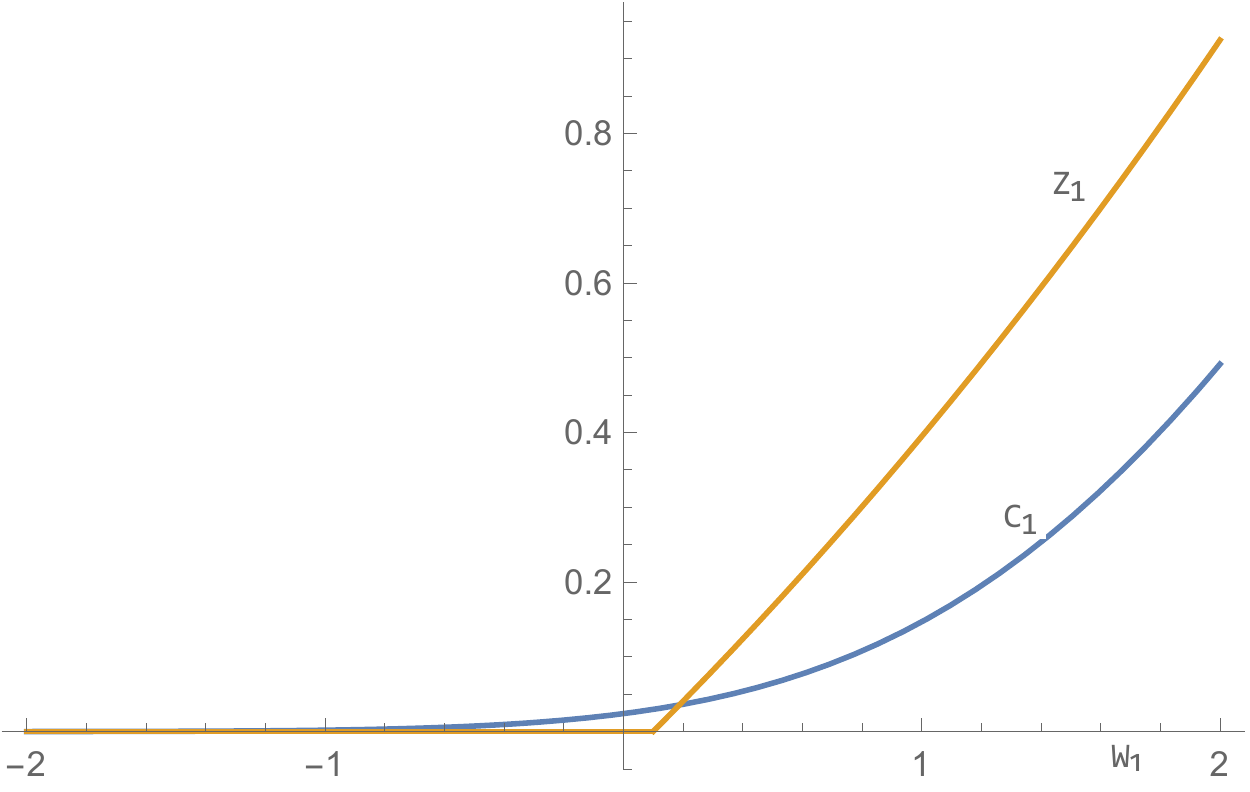}\includegraphics[width= 6cm]{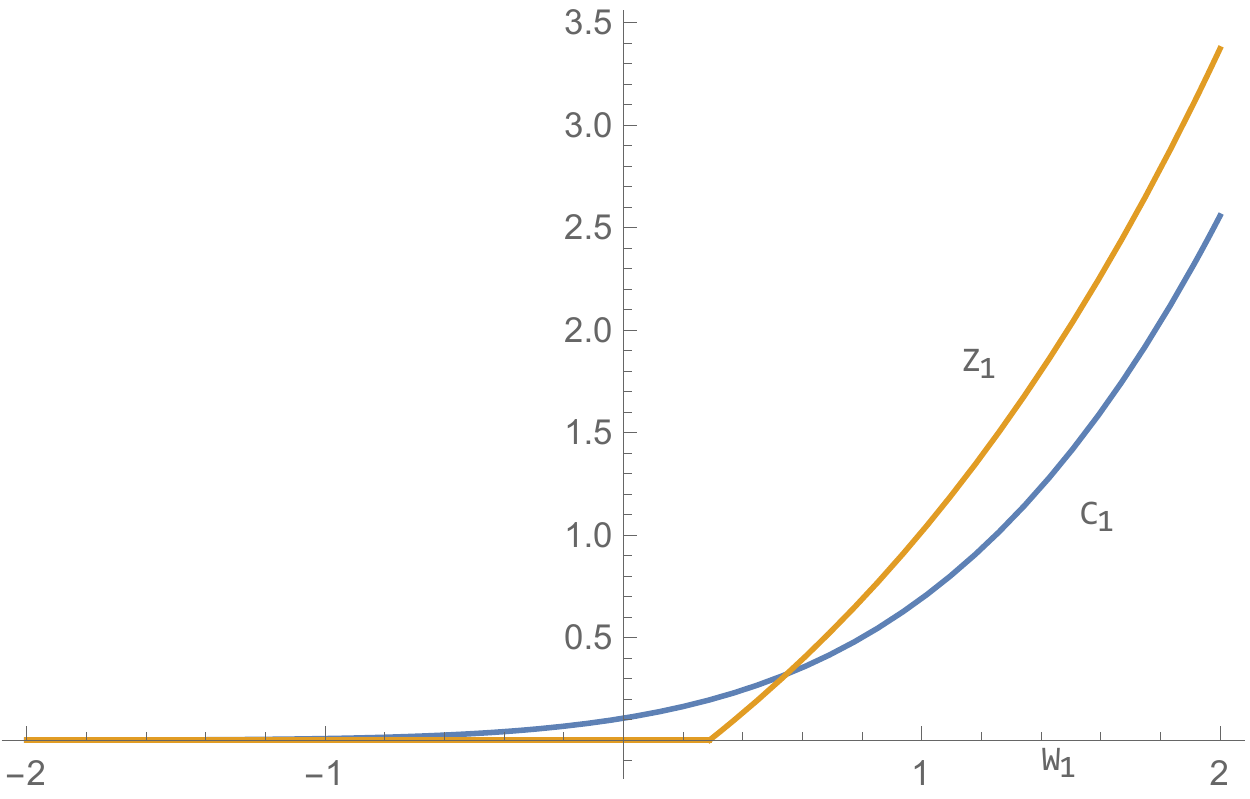}\caption{Cash-flow
$Z_{1}$ versus continuation value $C_{1}$ as a function of $W_{1}$ for (Pa1)
(left) and (Pa2) (right)}%
\label{exb}%
\end{figure}

\begin{figure}[tbh]
\hspace{3cm}%
\includegraphics[width= 7cm]{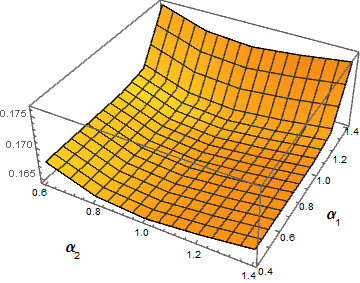}\caption{Object
function for BS-Call (Pa1) without randomization as function of $(\alpha
_{1},\alpha_{2})$}%
\label{plat}%
\end{figure}

\begin{figure}[tbh]
\hspace{3cm}%
\includegraphics[width= 7cm]{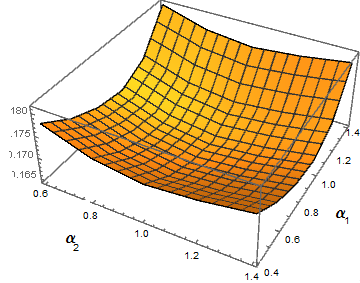}\caption{Object
function for BS-Call (Pa1) with optimal randomization as function of
$(\alpha_{1},\alpha_{2})$}%
\label{platran}%
\end{figure}

\begin{figure}[tbh]
\includegraphics[width= 6cm]{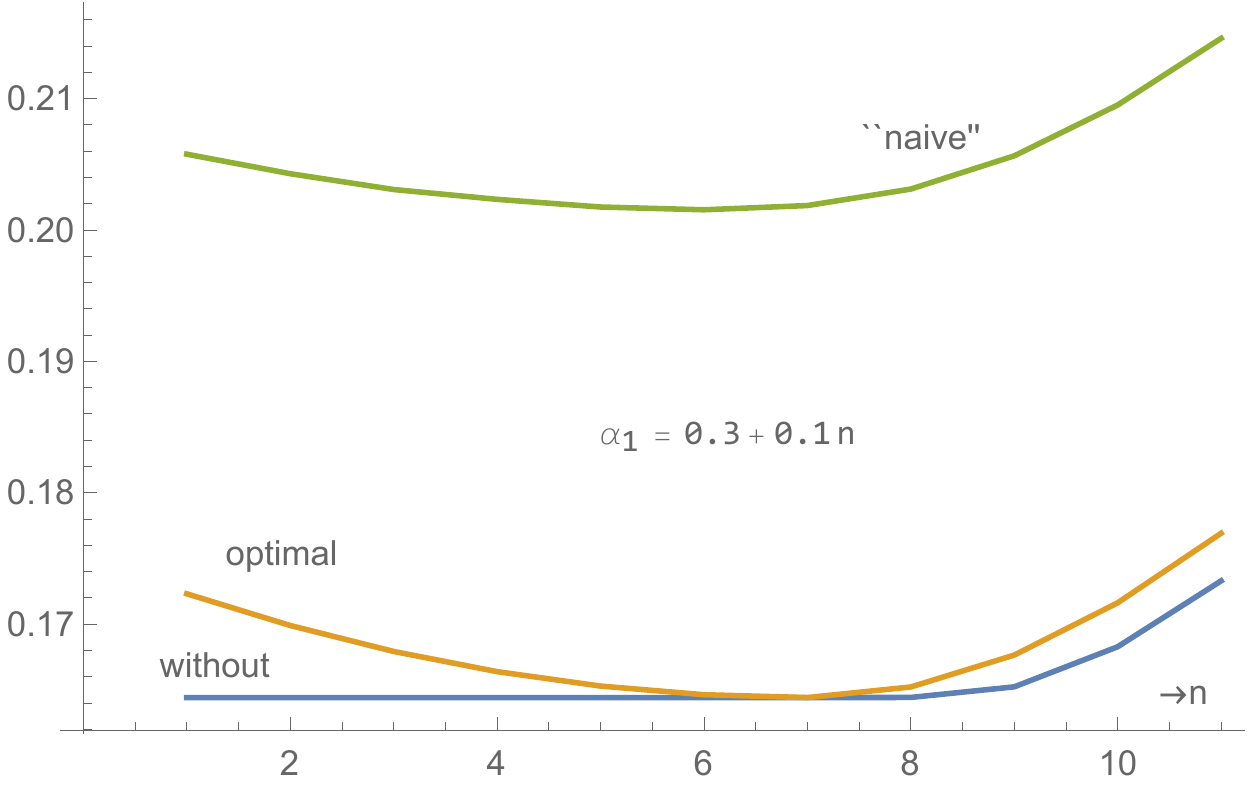}
\includegraphics[width= 6cm]{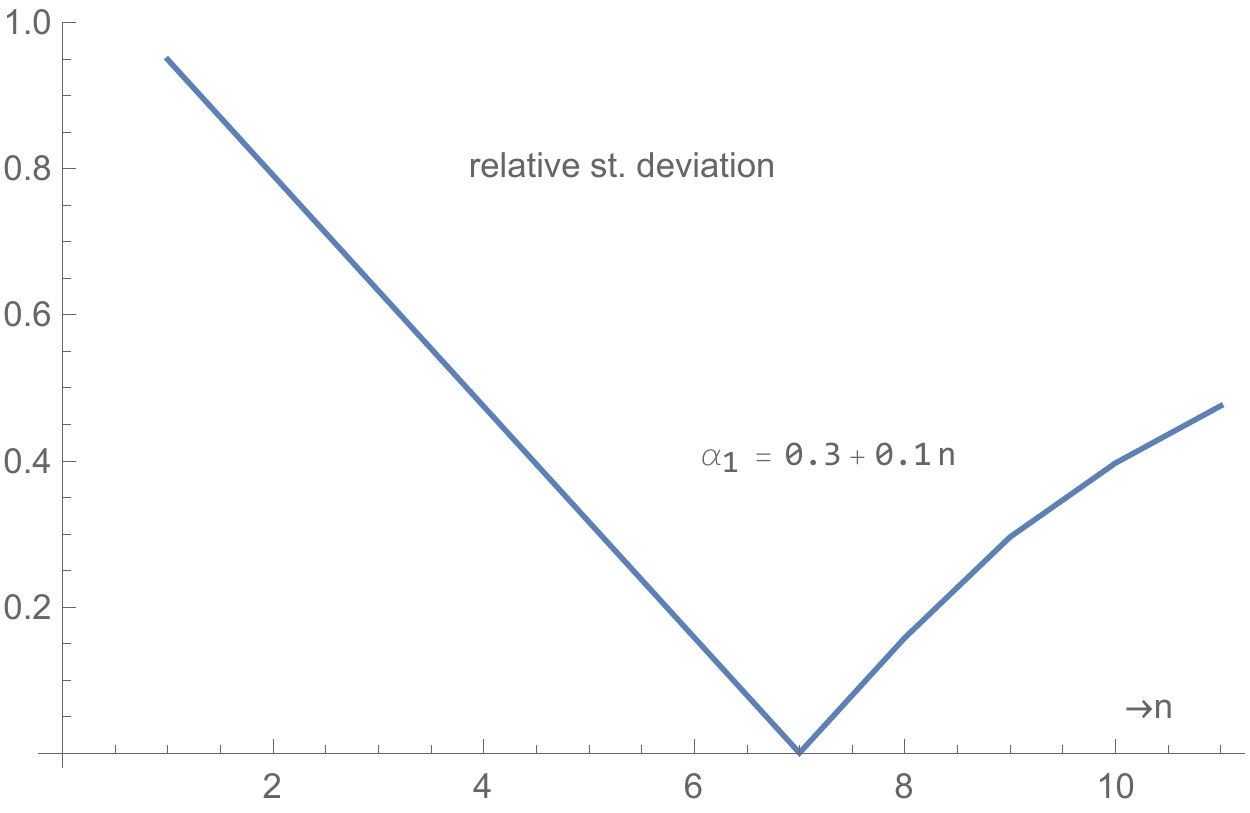}\caption{Left panel: object
functions of $\alpha_{1},$ with $\alpha_{2}=1$ fixed, for BS-Call (Pa1)
without, optimal, and ``naive'' randomization; right panel: relative deviation
of $\mathcal{Z}_{0}(\alpha_{1},1)$ (i.e. without randomization) }%
\label{ranvar}%
\end{figure}
\end{center}

\clearpage

\bibliographystyle{plain}
\bibliography{rand-stop_new}

\begin{thebibliography}{10}

\bibitem{J_AB2004}
Leif Andersen and Mark Broadie.
\newblock {A Primal-Dual Simulation Algorithm for Pricing Multi-Dimensional
  American Options.}
\newblock {\em Management Science}, 50(9):1222--1234, 2004.

\bibitem{J_Bel}
Denis Belomestny.
\newblock {Solving optimal stopping problems via empirical dual optimization}.
\newblock {\em The Annals of Applied Probability}, 23(5):1988--2019, 2013.

\bibitem{J_BelBenSch}
Denis Belomestny, Christian Bender, and John Schoenmakers.
\newblock {True upper bounds for {B}ermudan products via non-nested {M}onte
  {C}arlo}.
\newblock {\em Math. Finance}, 19(1):53--71, 2009.

\bibitem{J_BelHilSch}
Denis Belomestny, Roland Hildebrand, and John Schoenmakers.
\newblock Optimal stopping via pathwise dual empirical maximisation.
\newblock {\em Appl. Math. Optim.}, 79(3):715--741, 2019.

\bibitem{J_BrGl}
M.~Broadie and P.~Glasserman.
\newblock {A stochastic mesh method for pricing high-dimensional {A}merican
  options}.
\newblock {\em Journal of Computational Finance}, 7(4):35--72, 2004.

\bibitem{J_DK1994}
M.H.A. Davis and I.~Karatzas.
\newblock {A deterministic approach to optimal stopping.}
\newblock {Kelly, F. P. (ed.), Probability, statistics and optimisation. A
  tribute to Peter Whittle. Chichester: Wiley. Wiley Series in Probability and
  Mathematical Statistics. Probability and Mathematical Statistics. 455-466},
  1994.

\bibitem{J_DesFarMoa}
V.V. Desai, V.F. Farias, and C.C. Moallemi.
\newblock {Pathwise optimization for optimal stopping problems}.
\newblock {\em Management Science}, 58(12):2292--2308, 2012.

\bibitem{Gl}
Paul Glasserman.
\newblock {\em {Monte Carlo methods in financial engineering}}, volume~53.
\newblock Springer Science \& Business Media, 2003.

\bibitem{J_HK2004}
Martin Haugh and Leonid Kogan.
\newblock {Pricing American options: A duality approach.}
\newblock {\em Oper. Res.}, 52(2):258--270, 2004.

\bibitem{J_KS2006}
Anastasia Kolodko and John Schoenmakers.
\newblock {Iterative construction of the optimal Bermudan stopping time.}
\newblock {\em Finance Stoch.}, 10(1):27--49, 2006.

\bibitem{J_LS2001}
Francis~A. Longstaff and Eduardo~S. Schwartz.
\newblock {Valuing {A}merican options by simulation: a simple least-squares
  approach}.
\newblock {\em Review of Financial Studies}, 14(1):113--147, 2001.

\bibitem{J_Rogers2002}
Leonard C.~G. Rogers.
\newblock {{M}onte {C}arlo valuation of {A}merican options}.
\newblock {\em Mathematical Finance}, 12(3):271--286, 2002.

\bibitem{J_SchZhaHua}
John Schoenmakers, Jianing Zhang, and Junbo Huang.
\newblock {Optimal dual martingales, their analysis, and application to new
  algorithms for {B}ermudan products}.
\newblock {\em SIAM J. Financial Math.}, 4(1):86--116, 2013.

\bibitem{J_TV2001}
J.~Tsitsiklis and B.~{Van Roy}.
\newblock {Regression methods for pricing complex American style options.}
\newblock {\em IEEE Trans. Neural. Net.}, 12(14):694--703, 2001.

\end{thebibliography}

\end{document}